\documentclass[final,hidelinks,onefignum,onetabnum,dvipsnames]{siamart220329}

\usepackage{lipsum}
\usepackage{amsfonts}
\usepackage{graphicx}
\usepackage{epstopdf}
\usepackage{xcolor}
\ifpdf
  \DeclareGraphicsExtensions{.eps,.pdf,.png,.jpg}
\else
  \DeclareGraphicsExtensions{.eps}
\fi

\newsiamremark{remark}{Remark}
\newsiamremark{hypothesis}{Hypothesis}
\crefname{hypothesis}{Hypothesis}{Hypotheses}
\newsiamthm{claim}{Claim}
\newsiamthm{thm}{Theorem}

\headers{Multilinear Nystr\"{o}m in the Tucker format}{A. Bucci, L. Robol}

\title{A multilinear Nystr\"{o}m algorithm for low-rank  approximation 
  of tensors in Tucker format
  }

\author{Alberto Bucci\thanks{Department of Mathematics, University of Pisa, Italy 
  (\email{alberto.bucci@phd.unipi.it}, \email{leonardo.robol@unipi.it}).}
\and Leonardo Robol\footnotemark[2]}

\usepackage{amsopn}

\ifpdf
\hypersetup{
  pdftitle={An Example Article},
  pdfauthor={D. Doe, P. T. Frank, and J. E. Smith}
}
\fi

\usepackage[shortlabels]{enumitem}
\usepackage{verbatim}
\usepackage{mathrsfs,mathtools,amscd,tikz-cd,dirtytalk,comment,mathabx}
\usepackage{bm}
\usepackage{xcolor}
\usepackage{algorithm}
\usepackage{pgfplots}
\usetikzlibrary{pgfplots.groupplots}
\usepackage{algpseudocode}
\usepackage{tkz-euclide}
\usepackage{tikz}
\usepackage{siunitx}

\newcommand{\Xkk}{X_{\otimes \widecheck{k}}}

\newcommand{\Ukk}{U_{\otimes \widecheck{k}}}

\newcommand{\Omegakk}{\Omega_{\otimes \widecheck{k}}}
\newcommand{\fl}{\mathrm{f\kern.2ptl}}
\newcommand{\OO}{\textrm{\normalfont \O}}

\newcommand{\mlappr}{\mathcal T}
\renewcommand{\hat}{\widehat}
\pgfplotsset{compat=1.18}
\begin{document}

\maketitle

\begin{abstract}
The Nystr\"om method offers an effective way 
to obtain low-rank approximation of SPD matrices, 
and has been recently extended and analyzed
to nonsymmetric matrices (leading to the 
generalized Nystr\"om method). It is a  
randomized, single-pass, streamable, 
cost-effective, and accurate alternative to the randomized SVD, and it facilitates the computation of several matrix low-rank factorizations.

In this paper, we take these advancements a step further by introducing a higher-order variant of Nystr\"om's methodology tailored to approximating low-rank tensors in the Tucker format: the multilinear Nystr\"om technique. 
We show that, by introducing appropriate small modifications
in the 
formulation of the higher-order method, strong stability
properties can be obtained. This algorithm retains the key attributes of the generalized Nystr\"om method, positioning it as a viable substitute for the randomized higher-order SVD algorithm.
\end{abstract}



\section{Introduction}

Multilinear arrays, or tensors, offer a natural way to model higher-order 
structures, such as multivariate functions, models depending on several 
parameters, or high-dimensional PDEs. For this reason, they are often 
encountered in the numerical treatment of such applications. On one hand, 
this allows a seamless description of such structures. On the other 
hand, problems dealing with multi-dimensional arrays suffer the so-called \textit{curse of dimensionality} \cite{curse_of_dimensionality} that makes them computationally intractable due to memory requirements. More specifically, 
the storage requirements and the complexity of the numerical tools grow 
exponentially with the number of dimensions $d$. 

There have been several efforts to reduce the cost associated with 
dealing with tensors, and the most promising techniques leverage the use of 
low-rank properties \cite{grasedyck2013literature}. It turns out that giving an 
appropriate definition of tensor rank can be a challenging task. In contrast 
with the the $2$-dimensional case (i.e., when dealing with matrices),
where the definition of rank is essentially unique (and is linked with 
the dimensions of the subspaces spanned by rows and columns, which are the same), 
for tensors we have several alternatives.
In the matrix case, the rank and the closest 
rank $k$ approximant can be efficiently and stably computed 
with the singular value decomposition (SVD) \cite{svd}. 

The closest algebraic concept for $d > 2$ is 
finding the shortest decomposition in terms of outer products of $d$ vector 
(which we define as rank $1$ tensors). This is usually referred to as
canonical polyadic decomposition (CP or CPD) \cite{tensordecomp}, 
and  does not lead to favorable 
properties for numerical computations (for instance, the set of tensors with rank 
at most $k$ is not closed). An analogue of the SVD is not available in this 
setting. For this reason,  
it is convenient to look for other definitions of 
rank, such as the multilinear rank and the related Tucker 
decomposition (for which a higher order SVD is available \cite{hosvd})
or decompositions such as Tensor-Trains \cite{oseledets2011tensor} or 
Hierarchical Tucker \cite{grasedyck2010hierarchical}. 

In this work, we focus on the Tucker decomposition \cite{tucker} of low-rank tensors, 
and we address the problem of finding a 
randomized algorithm for the
low-rank approximation 
in this format exploiting only tensor mode-$j$ products, or contractions.

It has been recently shown that randomization can be a powerful and 
highly successful tool in numerical linear algebra \cite{kressner2023analysis,randomization_advantajes,nakatsukasa2021fast}, 
especially for large-scale problems where parallelization and limited access to data are needed. The main example of this situation is to find a near-optimal low-rank approximation to a matrix $A\in\mathbb{R}^{m\times n}$.

The underlying idea of randomized low-rank approximation 
algorithms is that the rows of a numerically low-rank matrix are almost linearly dependent and they can be embedded into a low-dimensional space without substantially altering their geometric properties. 
In particular, it has been observed that using random sketch matrices to construct a dimensionality reduction map (DRM) is often an efficient and nonadaptive way to achieve this.

In this regard, consider Algorithm~\ref{alg:HMT}, 
described by Halko, Martinsson, and Tropp in \cite{halko2011finding}, 
to determine an orthogonal matrix $Q\in \mathbb{R}^{m\times (r+\ell)}$ such that $Q Q^T A\approx A$ is a rank $r+\ell$ factorization of $A$.
\begin{algorithm}
\caption{Randomized rangefinder (HMT)} \label{algo:rangefinder}
\vspace{2mm}
\textbf{Input} \hspace{6mm} $A\in \mathbb{R}^{m\times n}$, rank $r\leq \min\{m, n\}$, oversampling parameter $\ell\geq 2$.\\
\textbf{Output} \hspace{3mm} $Q\in \mathbb{R}^{m\times (r+\ell)}$, with $Q^TQ = I$ and such that $Q Q^T A\approx A$.\\
\\
\phantom{\textbf{compute}} \hspace{1.5mm} Draw a random matrix $X\in \mathbb{R}^{n\times (r+\ell)}$;\\
\phantom{\textbf{compute}}  \hspace{1.5mm} Compute $AX$ (sketching);\\
\phantom{\textbf{compute}}  \hspace{1.5mm} Compute $Q$, orthogonal factor of an economy-size $QR$ of $AX$.

\vspace{2mm}
\label{alg:HMT}
\end{algorithm}

Despite its simplicity, this strategy is efficient and comes with attractive theoretical guarantees: it provides a near-optimal low-rank approximation to $A$.
When the matrix $X$ is chosen with
independent Gaussian distributed entries (with zero
mean and unit variance), one can prove the following upper 
bound for the expected value of the approximation error:
\[
\mathbb E \lVert {A - QQ^T A } \rVert_F \leq
\sqrt{
  1 + \frac{r}{\ell - 1}
} \cdot  \sqrt{
  \textstyle \sum_{j > r} \sigma_j^2(A)}.
\]
Since $\sqrt{\sum_{j > r} \sigma_j^2(A)}$ 
is the distance of $A$ from the 
set of rank $r$ matrices with respect to the Frobenius norm, the 
result is quasi-optimal up to a moderate constant. To better 
characterize the behavior of the randomized algorithm, 
one can describe the tail of the distribution. Under mild assumptions 
on $\ell$ (the oversampling parameter), we have:
\[
  \mathbb P \left\{ 
    \lVert {A - QQ^T A } \rVert_F > 
    \left(
      1 + \ell \sqrt{\frac{3r}{\ell+1}}
      + e \sqrt{2\ell (r+\ell) \log \ell} 
    \right)
    \sqrt{
  \textstyle \sum_{j > r} \sigma_j^2(A)}
\right\} \leq 3 \ell^{-\ell}.
\]
This and other results are discussed in \cite{halko2011finding} and the references therein.\footnote{To obtain the tail bound for the 
approximation error in the Frobenius norm, we used 
$t = \ell$ and $u = \sqrt{2 \ell \log \ell}$ in 
Theorem~10.7 of \cite{halko2011finding}.}

While very effective, the HMT method requires orthogonalization steps, which in specific situations may be relatively expensive, or not available because of 
the use of particular architectures (e.g., GPUs) \cite{nakatsukasa2020fast}. 
In addition, the method requires two passes, where the result of the first 
matrix-vector multiplication needs to be processed and then used in another 
matrix-vector multiplication. To reduce communication, single-pass 
algorithms are more attractive in several environments and have driven 
the interest in the design of ``streamable'' algorithms 
\cite{kressner2022streaming}. To mitigate this issue, one may rely on 
the Nystr\"om method; originally developed for SPD matrices, it has been recently extended to general matrices in \cite{tropp_GN}. In \cite{nakatsukasa2020fast}, where the method is called generalized Nystr\"om, or GN a few tricks to improve stability are proposed and a detailed 
error analysis of the method is given. 

The scheme of the GN method is the following: first, two DRM $X \in \mathbb{R}^{n\times r}$ and $Y\in \mathbb{R}^{m\times (r+\ell)}$, for some $\ell\geq 1$ are generated, then the low-rank approximation $\hat{A}$ is obtained by the formula
\begin{equation} \label{nystrom}
    \hat{A} = AX (Y^T A X)^{\dagger} Y^T A.
\end{equation}
The analysis in \cite{nakatsukasa2020fast} shows that, as for HMT, the quality of the approximation provided by GN is near-optimal and that the method can be implemented in a numerically stable fashion despite the presence of a (potentially)
ill-conditioned pseudoinverse.

In scenarios involving sparsity or structured data, 
one may wish to have factors with the same structure, a feature that GN might not have. In this context, 
alternative approaches such as the CUR factorization proposed in \cite{CUR_decomposition}, may be 
more appealing.

As regards tensors, several different randomized approaches for the Tucker approximation have been proposed in the literature \cite{randhosvdsurvey, power_tucker, other_tucker, minster_li_ballard, Saibaba_Minster_Arvind,  Another_Tucker}.
In \cite{caiafa2010generalizing,Tensor_hoid, oseledets2008tucker, HOID} for instance, the authors propose CUR-type algorithms for the decomposition of tensors. 

 Another straightforward approach is to replace every truncated SVD inside the HOSVD \cite{hosvd} or the STHOSVD \cite{sthosvd} algorithms, with the HMT method \cite{higherorderHMT}.
However, such methods and the variants mentioned above are not streamable
and may require (large and expensive) QR decompositions. 

In this work, we present an extension of GN to tensors, 
that recovers a Tucker decomposition; being a higher-order generalization 
of the Nystr\"om method, we refer to the new version as 
\textit{multilinear Nystr\"om} (MLN).

The resulting method has near-optimal approximation quality, 
and delivers results of comparable accuracy to other competing methods; the
computational cost is near-optimal for dense tensors, with small
hidden constants and, similarly to GN, it can be implemented in a numerically stable
fashion. The method avoids QR factorizations of large matrices, and only 
uses ``advanced'' linear algebra techniques (i.e., 
operations that are not matrix-vector products or matrix-matrix products) on small matrices. Hence, it is amenable to the implementation on 
various architectures, with minimal requirements. 

Theoretical results and numerical experiments show that the algorithm outperforms state-of-the-art methods both in terms of memory requirements, computational cost, and number of accesses to the original tensor data.

\section{Preliminary concepts and notation}

In this section, we introduce a few concepts and notations used
throughout the paper. A \textit{tensor} 
$\mathcal{A}\in\mathbb{R}^{I_1\times \dots \times I_d}$ is a $d$ dimensional array with entries
\[
a_{i_1i_2\dots i_d}, \quad 1\leq i_k \leq I_k, \quad  k = 1, \dots, d.
\]
The symbols used for transposition, Moore-Penrose inverse and 
Kronecker products of matrices are $T$, $\dagger$ and $\otimes$ respectively.
We use $\|\cdot\|_F$ for the Frobenius norm and $\|\cdot\|_2$ for
the spectral norm. 

We will repeatedly use the unfolding operation, which consists 
in reshaping tensors into matrices. The operation is sometimes 
called matricization or flattening. More specifically, 
when $\mathcal{A}$ is a tensor, its \textit{mode}-$k$ \textit{matricization} is denoted by $\mathcal{A}_{k}\in \mathbb{R}^{I_k \times \prod_{j\neq k} I_j}$ and satisfies:
\[
{(\mathcal{A}_{k})}_{i_k, j} = \mathcal{A}_{i_1,\dots, i_d},
\]
where
\[
j = 1 + \sum_{t=1, t \neq k}^d (i_t-1)J_t, \quad J_t = \prod_{s=1, s\neq k}^{t-1} I_s.
\]\\
The \textit{mode}-$k$ \textit{product} of a tensor $\mathcal{A}\in\mathbb{R}^{I_1\times \dots \times I_d}$ and a matrix $X \in \mathbb{R}^{J\times I_k}$ is denoted by $\mathcal{A}\times_k X$ and is such that

\[ 
 (\mathcal{A}\times_k X)_{i_1\dots i_{k-1}ji_{k+1}\dots i_d} = \sum_{s=1}^{I_k} \mathcal{A}_{i_1\dots i_{k-1}si_{k+1}\dots i_d}X_{js}.
\]
The mode-$k$ product along all dimensions (i.e., for $k = 1, \ldots, d$)
can be effectively expressed by leveraging a mix of 
matricizations and Kronecker products as follows:
\[
(\mathcal{A}\times_1 X_1 \times \dots \times X_d)_{k} = X_k \mathcal{A}_{k} (X_d \otimes \dots\otimes X_{k+1} \otimes X_{k-1} \otimes \dots \otimes X_1)^T.
\]
The focus of this work is on Tucker decompositions, which are
closely related with the concept of multilinear rank. Given a 
$d$-dimensional tensor, its multilinear rank is a tuple 
denoted by $rk^{ML}(\mathcal{A}) = (r_1, \dots, r_d)$, where $r_k$
is the matrix rank of $\mathcal A_k$, its mode-$k$ matricization \cite{hosvd}. 

When we need to compare multilinear ranks of different tensors we say that 
$rk^{ML}(\mathcal{A}) \leq rk^{ML}(\mathcal{B})$ if the multilinear rank of $\mathcal A$ is component-wise smaller than the one of $\mathcal B$. 

Several tensors of interest have some low-rank properties \cite{grasedyck2013literature}, which often 
only hold in an approximate sense (i.e., they are not low-rank, but 
they are close to a low-rank tensor in an appropriate metric). We 
introduce the class of $\epsilon$-approximable tensors, which makes this idea more precise. 
\begin{definition}\label{def:mlappr}
  Given a tuple $(r_1, \ldots, r_d)$ and $\epsilon > 0$, we define 
  $\mlappr_\epsilon(r_1, \ldots, r_d)$ as the set 
  \[
    \mlappr_{\epsilon}(r_1, \ldots, r_d) := \{ 
      \mathcal A \in \mathbb R^{n_1 \times \ldots \times n_d} \ | \ 
      \exists \mathcal \| \mathcal E \|_F \leq \epsilon \text{ with }
      rk^{ML}(\mathcal A + \mathcal E) \leq (r_1, \ldots, r_d)
    \}.
  \]
\end{definition}

When multiplying a tensor with mode-$k$ products along different 
modes (for instance for $k = 1, \ldots, d$), the order of the 
multiplication is irrelevant, since the operations commute. 
Hence, we often shorten formulas involving this kind 
of product as follows:
\begin{align*}
 \mathcal{A} \times_{i=1}^d X_i &= \mathcal{A}\times_1 X_1 \dots \times_d X_d, \\ 
X_{\otimes \widecheck{k}} &=X_d \otimes \dots\otimes X_{k+1} \otimes X_{k-1} \otimes \dots \otimes X_1.
\end{align*}

We also remark that the following properties hold, and are easily 
verified using the formulations of mode-$k$ products by means of 
matricization:
\[
\mathcal{A}\times_ k X_i \times_k X_j = \mathcal{A} \times_k X_j X_i, \qquad 
(\mathcal{A} \times_{i=1}^d X_i)_{k}=X_k \mathcal{A}_{k} \Xkk^T.
\]

We denote by $Q = \mathrm{orth}(X)$ the $Q$ factor of an economy
size QR factorization of a matrix $X$ with more rows than columns. 


\section{Randomized matrix low-rank approximation}
\label{sec:randomized}

The analysis of the multilinear Nystr\"om method (MLN) is based on results 
from the matrix case, which are briefly reviewed in this section. More 
specifically, we consider the approximant obtained by the HMT scheme 
for finding an orthogonal basis of the column span (as in \cite{halko2011finding})
and the GN scheme from \cite{nakatsukasa2020fast}.

Given a matrix $A$ of size $m\times n$ the approximants obtained by HMT and GN methods are given respectively by
\[
    \hat{A}_{HMT} = Q (Q^T A)\quad \text{and} \quad \hat{A}_{GN} = AX (Y^T A X)^{\dagger} Y^T A,
\]
where $Q = \mathrm{orth}(AX)$ and $X\in \mathbb{R}^{n\times r}$, $Y \in \mathbb{R}^{m\times (r+\ell)}$ are two DRM matrices.

We denote by $E_{HMT}$ the error of the approximation in Frobenius norm of the HMT approximant and with $E_{GN}$ that of GN; we have the following
upper bounds
\begin{align} \label{eq:HMT_error_bound}
E_{HMT} &\leq \|\widetilde{\Sigma}\|_F\sqrt{1 + \|\widetilde{V}_{\perp}^{T}X\|_2^2\|(\widetilde{V}^T X)^\dagger\|_2^2},\quad \\ \label{eq:GN_error_bound}
\quad E_{GN} &\leq E_{HMT}\sqrt{1 + \|Q_{\perp}^T Y\|_2^2\| (Q^T Y)^\dagger\|_2^2 },
\end{align}
where, for any $\hat{r}<r$, $\widetilde{\Sigma}$ is the diagonal term in the SVD of $A$ with a $0$ in place of the first $\hat{r}$ singular values and $\widetilde{V}$ is the orthogonal matrix with the first $\hat{r}$ right singular vectors of $A$. 

The term $\|\widetilde \Sigma\|_F$ is the optimal error that would be obtained 
by a truncated SVD; hence, it is clear that it is important to choose the 
DRM in a way that makes the other terms as small as possible (with high 
probability). At the same time, we wish to maintain the cost of taking
the matrix-vector products small, so it makes sense to use DRM drawn from 
a set of structured matrices which have fast matrix-vector products
routines available.

A few choices of random samplings that allow for fast multiplications 
arise from subsampling trigonometric transforms. Examples include the
Subsampled Randomized Hadamard Transform (denoted by SRHT) \cite{boutsidis2013improved, tropp2011improved}, 
and
the Subsampled Randomized Fourier Transform (SRFT) \cite{rokhlin2008fast}. These approaches reduce the 
cost of forming $AX$ to $\mathcal{O}(mn \log n)$, where $n$ is the number of 
columns of $X$, and $m$ the number of rows of $A$. The theory for these subsampled transforms can be more complex than the one for more ``classical'' choices, such as Gaussian
matrices; the latter are deeply understood and have sharp error bounds available (see \cite{randomization_advantajes} and the references therein).

The use of GN has a few advantages with respect to the HMT scheme: 
it avoids orthogonalizations and can be used as a single-pass approximation
method.

However, without a proper implementation, the stability of GN can be cause for concern. The pseudocode in Algorithm~\ref{alg:GN}
reports the implementation suggested in \cite{nakatsukasa2020fast}. 

\begin{algorithm}
\caption{Generalized Nystr\"om (GN)} 
\vspace{2mm}
\textbf{Input} \hspace{6mm} $A\in \mathbb{R}^{m\times n}$, rank $r\leq \min\{m, n\}$, oversampling parameter $\ell\geq 1$.\\
\textbf{Output} \hspace{3mm} Low-rank approximant $\hat{A}$ of $A$.\\
\\
\phantom{\textbf{compute}} \hspace{1.5mm} Draw two random matrices $X\in \mathbb{R}^{n\times r}$ and $Y\in \mathbb{R}^{m\times (r+\ell)}$;\\
\phantom{\textbf{compute}}  \hspace{1.5mm} Compute $AX$, $Y^TA$; \\ 
\phantom{\textbf{compute}}  \hspace{1.5mm} Compute an economy-size QR factorization $Y^TAX = ZR$;\\
\phantom{\textbf{compute}}  \hspace{1.5mm} Compute $\hat{A} = ((AX)R^{-1})(Z^T(Y^TA))$.
\vspace{2mm}
\label{alg:GN}
\end{algorithm}

In practice, a slight oversample parameter $\ell$ makes this implementation stable, but no theoretical assessments have been conducted. While stability cannot be established
for \eqref{nystrom} as is, there is an inexpensive modification that guarantees stability:
\begin{equation} \label{stabilized_nystrom}
\hat{A} = AX (Y^TAX)^{\dagger}_\epsilon Y^TA,
\end{equation}
which is the stabilized generalized Nystr\"om (SGN) method.

Here $(Y^TAX)_\epsilon$ denotes the $\epsilon$-pseudoinverse, that is if 

\[Y^TA X =\begin{bmatrix}
    U_1 & U_2
\end{bmatrix}\begin{bmatrix}
    \Sigma_1 & \\
    & \Sigma_2
\end{bmatrix}\begin{bmatrix}
    V_1 & V_2
\end{bmatrix}^T\]
is the SVD, where $\Sigma_1$ contains singular values larger than $\epsilon$, then $(Y^TAX)^\dagger_\epsilon = V_1 \Sigma_1^{-1} U_1^T$. 
In the following, $\epsilon$ will be chosen as
a modest multiple of the unit roundoff $u$ times $\|A\|_F$, that is $\epsilon = \mathcal{O}(u\|A\|_F)$.

Different strategies to implement SGN in a stable manner can be found in \cite{nakatsukasa2020fast}.

\section{Multilinear Nystr\"om}\label{sec: mlnystrom}

This section is devoted to extending GN to tensors, and, 
in particular, lay the ground for determining if the stability analysis and the 
related guarantees that are explored in \cite{nakatsukasa2020fast}
for matrices carry over to the tensor setting. 

To simplify the analysis, it is convenient to rewrite the classical 
Nystr\"om approximant \eqref{nystrom}
in the matrix setting in a slightly different
way. Leveraging the properties of the Moore-Penrose pseudoinverse we have 
the following identity:
\begin{equation*} 
 AX(Y^TAX)^\dagger Y^TA = 
   A X ( Y^T A X)^{\dagger}Y^T A X( Y^T A X)^{\dagger}Y^TA.
\end{equation*}
The above reformulation identifies a structure in the approximant, which is formed 
by a small core matrix $Y^TAX$ and two matrices of the form $AX(Y^TAX)^{\dagger}$ and $(Y^TAX)^{\dagger}Y^TA = (A^TY(X^TA^TY)^\dagger)^T$ which invert the sketching procedure. Notice that the matrices $\mathcal{P}_1:=AX(Y^TAX)^{\dagger}Y^T$ and $\mathcal{P}_2:=A^TY(X^TA^TY)^\dagger X^T$ are approximate 
oblique projections on the column space of $A$ and $A^T$ respectively.

This formulation can be replicated in the tensor setting. Given a tensor $\mathcal{A}\in\mathbb{R}^{n_1\times \dots \times n_d}$ and sketch matrices $X_i\in \mathbb{R}^{n_i\times r_i} \quad i= 1, \dots, d$ \quad we form a core tensor $\mathcal{A}\times_{i=1}^d X_i^T$, by sketching $\mathcal{A}$ in all the modes and we construct the matrices $(\mathcal{A}\times_{i\neq k} X_i^T)_{k} (\mathcal{A}\times_{i=1}^d X_i^T )_{k}^{\dagger}$ that give an approximate basis of the column span of the 
mode-$k$ matricization. 

This leads to a randomized algorithm for finding a low-rank representation of a tensor in Tucker format. Formally, we define the approximant $\widehat{\mathcal{A}}$ of $\mathcal{A}$ by
\begin{equation} \label{mlnystrom}
\begin{split}
  \widehat{\mathcal{A}} &= (\mathcal{A}\times_{i=1}^d X_i^T)\times_{k=1}^d \left((\mathcal{A}\times_{i\neq k} X_i^T)_{k} (\mathcal{A}\times_{i=1}^d X_i^T )_{k}^{\dagger}\right)\\ 
  &= (\mathcal{A}\times_{i=1}^d X_i^T)\times_{k=1}^d \left(
    \mathcal A_k \Xkk 
    (X_k^T \mathcal A_k \Xkk)^\dagger
  \right).
  \end{split}
\end{equation}
Note that when $d=2$ the matrix $\widehat{A}_1$, flattening along the first index of \eqref{mlnystrom}, is indeed the reformulation of the generalized Nystr\"om approximant proposed 
above. In fact, we have
\begin{equation}\label{2mode_example}
\widehat{A}_{1}=AX_2(X_1^T A X_2)^{\dagger}(X_1^T A X_2)(X_1^T A X_2)^{\dagger}X_1^TA=AX_2(X_1^T A X_2)^{\dagger}X_1^T A.
\end{equation}
It is convenient to define, for each $k=1, \dots, d$ the projection matrix  
\begin{equation} \label{projection}
\mathcal{P}_k:=\left((\mathcal{A}\times_{i\neq k} X_i^T)_{k} (\mathcal{A}\times_{i=1}^d X_i^T )_{k}^{\dagger}\right)X_k^T= \mathcal{A}_k\Xkk(X_k^T\mathcal{A}_k\Xkk)^\dagger X_k^T.
\end{equation}
The matrix $\mathcal P_k$ is an oblique projection onto the column space of $\mathcal{A}_k$ and allows us to rewrite \eqref{mlnystrom} in the compact form
\[
\hat{\mathcal{A}} = \mathcal{A} \times_{k=1}^d \mathcal{P}_k.
\]
Note that the formula we mentioned above is not new, and was already present in Caiafa and Cichocki \cite{caiafa2014stable}. 
This paper was published before the formulation of GN 
in \cite{tropp_GN, nakatsukasa2020fast}. To the best of our knowledge, 
the connection between these two works has not been emphasized so far.

An important feature of \eqref{mlnystrom} is that when the tensor $\mathcal{A}$ is expressed in Tucker format, that is $\mathcal{A} = \mathcal{C}\times_{k=1}^d U_k$ with small core tensor $\mathcal{C}$, the cost of computing the different mode products drops dramatically: we can first compute the products $\Psi_k = X_k^T U_k$ and then the mode products between $\mathcal{C}$ and the $\Psi_k$. 
This suggests that \eqref{mlnystrom} can be used as an effective method for recompression of tensors in Tucker format. 

For the matrix case ($d = 2$) in the formulation 
of \eqref{stabilized_nystrom}, the key property that allows to prove 
some form of stability under floating point inaccuracies is to have unbalanced 
dimensions in $X$ and $Y$; the fact that $Y$ has $r + \ell$ rows and 
$X$ only $r$ plays the same role of
the oversampling in Algorithm~\ref{alg:HMT}, 
and allows to stabilize
the least square solution. This property is lost when we factor out the projections as in \eqref{mlnystrom}. Hence, the 
reformulation of \eqref{mlnystrom} has two opposite effects on the
design of the low-rank approximation scheme:
\begin{itemize}
    \item On one hand, it makes extending the approach to $d > 2$ much 
      easier, because it suffices to define the oblique projectors 
      $\mathcal P_k$ and apply them on all modes from $k = 1, \ldots, d$.
    \item On the other hand, it makes a stability analysis more 
      difficult (or impossible without further modifications), because any 
      immediate bound will depend on the norms of $\mathcal P_k$, which
      are hard to control.
\end{itemize}

We have already discussed how the reformulation allows for an easy extension
for a generic $d$. In the next section, we will show that we can introduce 
further degrees of freedom in the choice of samplings, and this will greatly help in slightly modifying the method to make it stable. The 
core idea is that, instead of fixing only $d$ sampling matrices 
$X_1, \ldots, X_d$, we can choose $2d$ by introducing additional
samplings $Y_1, \ldots, Y_d$. This will allow us to reintroduce the unbalanced 
dimension (and the stabilized least square solvers) into the picture. 
\subsection{Introducing more sketching matrices}

This section introduces a small variant of the sketching procedure
described in \eqref{mlnystrom}, that allows us to select more than 
$d$ sketching matrices ($2d$ instead of $d$). The aim of this 
generalization is to design an approximation method with 
better stability properties. 

To make it easier to follow the discussion, we start by showing
how this generalization can be formulated in the matrix case. 

Recall that, according to the previous discussion, the approximation provided by GN can be written as
\begin{equation*}
A X ( Y^T A X)^{\dagger}Y^T A X( Y^T A X)^{\dagger}Y^TA.  
\end{equation*}

Instead of using the same $X$ and $Y$ for the two projections $A X ( Y^T A X)^{\dagger}Y^T$ and $X( Y^T A X)^{\dagger}Y^TA$, we may achieve more generality by using different sketching matrices ($X_1\in\mathbb{R}^{n_1\times r_1}$ and $Y_1\in \mathbb{R}^{m_1\times (r_1+\ell_1)}$ for the first projection and $X_2\in\mathbb{R}^{n_2\times r_2}$ and $Y_2\in \mathbb{R}^{m_2\times (r_2+\ell_2)}$ for the second), still ensuring that whenever $A$ is of low-rank, an exact representation is retrieved (at least theoretically, if no floating point errors are 
considered). The resulting approximant would be as follows:
\begin{equation}\label{extended_nystrom}
\hat{A} = AX_1 (Y_1^T A X_1)^{\dagger}Y_1^T A Y_2(X_2^T A Y_2)^{\dagger} X_2^T A.
\end{equation}

Note that in the second projection, not only have we substituted $X$ and $Y$ with $X_2$ and $Y_2$, but we have also swapped them for reasons of symmetry.

In addition, by looking at $\widehat{A}$ as a 2-dimensional tensor, this approximation can be obtained by applying oblique 
projectors $\mathcal P_1$ and $\mathcal P_2$ along the modes $1$ and $2$, respectively, 
with the projectors $\mathcal P_1 = AX_1 (Y_1^T A X_1)^{\dagger}Y_1^T$ and $\mathcal P_2 = A^TX_2 (Y_2^T A^T X_2)^{\dagger}Y_2^T$. 
We use the same notation $\mathcal P_k$ used 
in \eqref{projection} for these 
more general projectors, since there is no risk 
of confusion. 

Going back to the notation of mode-$j$ products, we may rewrite the 
approximation as follows:
\[
\hat{A} = A\times_1 \mathcal P_1 \times_2 \mathcal P_2.
\]
In the matrix case $d = 2$, this idea is not particularly appealing, 
since it requires more matrix-vector products. 

For the ``standard'' GN, only one of two projections needs to be 
computed, since we may write:
\[
\hat{A} = A\times_1 \mathcal{P}_1 \times_2 \mathcal{P}_2 = A\times_1 \mathcal{P}_1 = A\times_2 \mathcal{P}_2.
\]
The idea will however be valuable in the tensor setting, where we 
have already anticipated that the simplification of the projections
that happens for $d = 2$ is not easily obtainable. The next section 
is devoted to analyzing \eqref{extended_nystrom} in detail; we anticipate that its accuracy is just slightly less than GN, but still near-optimal, its cost is about twice as expensive as GN and crucially, can be implemented in a numerically stable fashion.

This more general approximant is expressed in a form 
that can be readily extended 
to the tensor setting. Given a tensor $\mathcal{A}\in \mathbb{R}^{n_1\times \dots \times n_d}$ and sketch matrices $X_i\in \mathbb{R}^{\prod_{j\neq i}n_j\times r_i}$ and $Y_i\in \mathbb{R}^{n_i\times (r_i +\ell_i)}$, $i=1,\dots, d$, we define the \textit{multilinear Nystr\"om} (MLN) approximant  $\hat{\mathcal{A}}$ of $\mathcal{A}$ as follows:
    \begin{equation} \label{extended_mlnystrom}
\widehat{\mathcal{A}}  =   (\mathcal{A}\times_{k=1}^d Y_k^T)\times_{k=1}^d \mathcal{A}_k X_k(Y_k^T\mathcal{A}_k X_k)^{\dagger}.
\end{equation}
The approximant is obtained by sketching with matrices $Y_k$, and 
using oblique projections built with the matrices $X_k$ in order 
to construct a low-rank Tucker factorization. The oblique projectors 
are easy to write explicitly: we have $\mathcal{P}_k := \mathcal{A}_k X_k(Y_k^T\mathcal{A}_k X_k)^{\dagger}Y_k^T$, and this will be 
helpful for our analysis later on. 
In equation \eqref{extended_mlnystrom} we assumed $Y_k$ with more columns than $X_k$. Nevertheless, it is worth noting that by relaxing this assumption, the resulting reformulation is an extension of the
original one in \eqref{mlnystrom}, 
which can be recovered by substituting
$Y_k$ with $X_k$ and $X_k$ with $\Xkk$ 
in \eqref{extended_mlnystrom}.

The pseudocode in Algorithm~\ref{alg:MLN} describes the method for computing the approximant in Equation $\eqref{extended_mlnystrom}$. 
The QR factorization in 
the pseudocode is the standard economy-size QR factorization without pivoting.

We remark that even though we obtained formula \ref{extended_mlnystrom} with a different approach, the MLN approximant turns out to be mathematically equivalent to the one provided by Algorithm 4.3 in \cite{Another_Tucker}. 
There are however algorithmic differences: our implementation avoids the expensive QR factorization of $A_kX_k$, and, as we prove in Section \ref{sec:stability}, has the same stability properties of GN.

\begin{algorithm}
\caption{Multilinear Nystr\"om (MLN)} 
\vspace{2mm}
\textbf{Input} \hspace{6mm} $\mathcal{A}\in \mathbb{R}^{n_1\times\dots \times n_d}$, multilinear rank $r = (r_1, \dots, r_d) \leq (n_1,\dots, n_d)$, \\
\phantom{\textbf{Input}} \hspace{6mm} oversampling vector $\ell= (\ell_1,\dots, \ell_d)$.\\
\textbf{Output} \hspace{3mm} Low-rank Tucker approximant $\hat{\mathcal{A}}$ of $\mathcal{A}$.\\
\\
\phantom{\textbf{Input}} \hspace{6mm} \textbf{for} $k = 1, \dots, d$\\
\mbox{}~~~~~~~~~~~~~~~~~~~Draw random matrices $X_k\in \mathbb{R}^{\prod_{i\neq k} n_i\times r_k}$ and $Y_k\in \mathbb{R}^{n_k\times (r_k+\ell_k)}$;\\
\mbox{}~~~~~~~~~~~~~~~~~~~Compute $\mathcal{A}_k X_k$, $Y_k^T\mathcal{A}_k$; \\ 
\mbox{}~~~~~~~~~~~~~~~~~~~Compute an economy-size QR
factorization $Y_k^T\mathcal{A}_k X_k = Z_k R_k$;\\
\mbox{}~~~~~~~~~~~~~~~~~~~Compute $\hat{\mathcal{A}} = ((\mathcal{A}\times_{k=1}^d Y_k^T)\times_{k=1}^d Z_k^T)\times_{k=1}^d (\mathcal{A}_k X_k R_k^{-1}).$ \\
\phantom{\textbf{Input}} \hspace{6mm} \textbf{end}
\vspace{2mm}
\label{alg:MLN}
\end{algorithm}

\section{Properties of MLN}

In this section, we prove that the accuracy of MLN is near-optimal and that the version of Algorithm~\ref{alg:MLN} with the $\varepsilon$-pseudoinverse, stabilized multilinear Nystr\"om (SMLN), 
may be implemented in a stable way. 

Concerning accuracy, our objective is to show that, choosing appropriate
sketchings, the performances attained by MLN are close to the one of RHOSVD (a tensor version of HMT), 
which is in turn close to the HOSVD with high probability \cite{higherorderHMT}. 

For what concerns the stabilization, we will show how some ideas from 
\cite{nakatsukasa2020fast} can be generalized from the matrix to the tensor setting, allowing stronger stability guarantees. This will be 
obtained thanks to the particular choices of sketchings that we made 
in the previous section, introducing the matrices $Y_k$, and is not easy 
to obtain otherwise (such as in the first multilinear Nystr\"om proposed by Caiafa and Cichocki \cite{caiafa2014stable}, where some form of stability 
has been shown only in the matrix case $d = 2$). 

\subsection{Accuracy of MLN}

In order to prove the results related to the accuracy of the MLN scheme, 
we now introduce a few auxiliary lemmas. 

We denote by $\mathcal{P}_k$ the oblique projections $\mathcal P_k := \mathcal{A}_k X_k(Y_k^T\mathcal{A}_k X_k)^{\dagger}Y_k^T$. In this way, we may write the approximation obtained by MLN with the compact notation 
$\hat{\mathcal{A}} = \mathcal{A}\times_{k=1}^d \mathcal{P}_k$.

\begin{lemma}\label{lemma_1}
Let $\mathcal A$ be a $d$-dimensional tensor, and 
$\mathcal P_k := \mathcal{A}_k X_k(Y_k^T\mathcal{A}_k X_k)^{\dagger}Y_k^T$ 
for sketching matrices $X_k, Y_k$ of compatible dimensions. Then, the following inequality holds
\begin{equation} \label{sum_formula}
\| \mathcal{A} - \widehat{\mathcal{A}} \|_F \leq \displaystyle\sum_{k=1}^{d} \|\mathcal{A}\times_{i=1}^{k-1}\mathcal{P}_i - \mathcal{A}\times_{i=1}^k\mathcal{P}_i\|_F.
\end{equation}
\end{lemma}
\begin{proof} 
We expand the approximation error
$\mathcal A - \hat{\mathcal A}$ in 
the telescopic sum 
\[
  \mathcal A - \hat{\mathcal A} = 
    \sum_{k = 1}^d \left( \mathcal{A}\times_{i=1}^{k-1}\mathcal{P}_i - \mathcal{A}\times_{i=1}^k\mathcal{P}_i \right).
\]
The result follows by taking Frobenius
norms on both sides and using the sub-additivity property.
\end{proof}

The terms in the summation above are of the form 
$\| \mathcal{B} - \mathcal{B} \times_k \mathcal P_k \|_F$, where $\mathcal{B} = \mathcal A \times_{i = 1}^{k-1} \mathcal P_i$. This fact allows us to relate the 
approximation error of MLN with the approximation error of a GN obtained 
by appropriately flattening the tensors. 

The next lemma is a step in this direction: it relates the 
approximation error along mode-$1$ with the error in the projection 
used with HMT on the matricization. The result is in line with 
the one proved in \cite{nakatsukasa2020fast} for the matrix case 
when analyzing GN. 

\begin{lemma}\label{lemma_2}
   Let $\mathcal{A}\in \mathbb{R}^{n_1\times \dots \times n_d}$, $X_1\in \mathbb{R}^{\prod_{k\neq 1} n_k \times r_1}$ and let $Q_1 R_1$ denote an economy size QR of 
   $\mathcal{A}_1 X_1\in \mathbb{R}^{{n_1}\times r_1}$. If $Y_1$ has $r_1 + \ell_1$ 
   columns and $Y_1^T Q_1$ has rank $r_1$ then
\begin{equation}
\| \mathcal{A} - \mathcal{A}\times_1 \mathcal{P}_1 \|_F \leq \|Q_{1\perp}^T\mathcal{A}_1\|_F\sqrt{1 + \| (Y_1^T Q_1)^\dagger\|_2^2 \|Y_1^TQ_{1\perp}\|_2^2}.
\end{equation}
\end{lemma}

\begin{proof}
  The Frobenius norm of a tensor coincides with the one of any 
  of its matricizations; hence, we may write 
  \begin{align*}
      \|\mathcal{A} - \mathcal{A} \times_1 \mathcal{P}_1\|_F&= \|\mathcal{A}_1 - Q_1Q_1^T\mathcal{A}_1 + Q_1Q_1^T \mathcal{A}_1 - \mathcal{A}_1X_1 (Y_1^T\mathcal{A}_1 X_1)^\dagger Y_1^T \mathcal{A}_1\|_F \\
      & =\|\mathcal{A}_1 - Q_1Q_1^T\mathcal{A}_1 + Q_1(Q_1^T -(Y_1^T Q_1)^\dagger Y_1^T)\mathcal{A}_1 \|_F
  \end{align*}
  Let us denote with $Q_{1\perp}$ a matrix whose columns span the 
  orthogonal space to the columns of $Q_1$. Then, the following 
  two identities hold:
  \begin{align}
    Q_1Q_1^T + Q_{1\perp}Q_{1\perp}^T &= I, \label{eq:id1} \\
    (Q_1^T - (Y_1^TQ_1)^\dagger Y_1^T) Q_1 &= 0. \label{eq:id2} 
  \end{align}
The first relation is simply the decomposition of the identity as the 
projections over the column span of $Q_1$ and its orthogonal space; 
the second follows from our assumption that $Y_1^TQ_1$ is of full 
column-rank. We make use of these two identities to further 
simplify the previous expression for $\|\mathcal{A} - \mathcal{A} \times_1 \mathcal{P}_1\|_F$:
\begin{align*}
     \|\mathcal{A} - \mathcal{A} \times_1 \mathcal{P}_1\|_F&
     =\|\mathcal{A}_1 - Q_1Q_1^T\mathcal{A}_1 + Q_1(Q_1^T -(Y_1^T Q_1)^\dagger Y_1^T)\mathcal{A}_1 \|_F
     \\ 
     \text{\scriptsize [\eqref{eq:id1} + \eqref{eq:id2}]} \rightarrow &=  \|Q_{1\perp}Q_{1\perp}^T\mathcal{A}_1 + Q_1(Q_1^T -(Y_1^T Q_1)^\dagger Y_1^T)Q_{1\perp}Q_{1\perp}^T\mathcal{A}_1 \|_F
        \\
{\text{\scriptsize $Q_1^T Q_{1 \perp} = 0$}} \rightarrow &=\| Q_{1\perp}Q_{1\perp}^T\mathcal{A}_1 -Q_1(Y_1^T Q_1)^\dagger (Y_1^TQ_{1\perp})Q_{1\perp}^T\mathcal{A}_1 \|_F, 
\end{align*}
To conclude, 
recall that whenever two matrices $A, B$ have orthogonal columns, 
we have $\| A + B \|_F^2 = \| A \|_F^2 + \| B \|_F^2$. Therefore, 
\begin{align*}
    \|\mathcal{A} - \mathcal{A} \times_1 \mathcal{P}_1\|_F^2 &=
    \| Q_{1\perp}Q_{1\perp}^T\mathcal{A}_1\|_F^2 + \| Q_1(Y_1^T Q_1)^\dagger (Y_1^TQ_{1\perp})Q_{1\perp}^T\mathcal{A}_1 \|_F^2 \\ 
    & \leq {\|Q_{1\perp}^T\mathcal{A}_1\|_F^2 + \| (Y_1^T Q_1)^\dagger\|_2^2 \|Y_1^TQ_{1\perp}\|_2^2 \|Q_{1\perp}^T\mathcal{A}_1 \|_F^2}\\
     & \leq \|Q_{1\perp}^T\mathcal{A}_1\|_F^2 \cdot \left({1 + \| (Y_1^T Q_1)^\dagger\|_2^2 \|Y_1^TQ_{1\perp}\|_2^2}\right).
\end{align*}

Taking the square root on both sides of the inequality
gives us the claim. 
\end{proof}

Lemma~\ref{lemma_2} is stated for $\mathcal P_1$, but clearly holds for 
any $\mathcal P_i$ with $i = 1, \ldots, d$ up to permuting the indices. 
Hence, a rough bound for the accuracy of the low-rank approximation may 
be obtained by bounding the terms in \eqref{sum_formula} as follows:
\[
 \|\mathcal{A}\times_{i=1}^{k-1}\mathcal{P}_i - \mathcal{A}\times_{i=1}^k\mathcal{P}_i\|_F \leq \| \mathcal{A} - \mathcal{A}\times_k \mathcal{P}_k\|_F \prod_{i=1}^{k-1} \|\mathcal{P}_i\|_2. 
\]
We may then proceed by finding upper bounds for 
$\| \mathcal{A} - \mathcal{A}\times_k \mathcal{P}_k\|_F$ and the norms of the projections $\mathcal{P}_k$ separately. However, as discussed in \cite[Section 3.3]{nakatsukasa2020fast}, this would lead to a large overestimate.

To obtain more predictive bounds, we follow another approach 
and consider the tensor $\mathcal{A}\times_{i=1}^k\mathcal{P}_i$ as the tensor $\mathcal{A}\times_{i=1}^{k-1}\mathcal{P}_i$ projected along the $k$th mode.
This yields the following result. Since the proof follows similar 
steps to the one of \ref{lemma_2}, some details are omitted.

\begin{lemma}\label{lemma_3}
    Let $\mathcal{A}\in \mathbb{R}^{n_1\times \dots \times n_d}$, $X_k\in \mathbb{R}^{\prod_{i\neq k} n_i \times r_k}$ and let $Q_k R_k$ denote an economy size QR of 
   $\mathcal{A}_k X_k\in \mathbb{R}^{{n_k}\times r_k}$.
    If $Y_k$ has $r_k + \ell_k$ columns and $Y_k^T Q_k$ has rank 
    $r_k$, then setting
    $E_k := \|\mathcal{A}\times_{i=1}^{k-1} \mathcal{P}_i-\mathcal{A}\times_{i=1}^{k} \mathcal{P}_i\|_F$
    we have
    \begin{align*}
    E_k
    &\leq
     \Big(\|Q_{k\perp}^T\mathcal{A}_k\|_F+ 
     \|\mathcal{A}- \mathcal{A}\times_{i=1}^{k-1}\mathcal{P}_i\|_F
     \Big)\sqrt{1 + \| (Y_k^T Q_k)^\dagger\|_2^2 \|Y_k^TQ_{k\perp}\|_2^2}.
    \end{align*}
\end{lemma}

\begin{proof}
    Let us define $B_1 = I$ and $B_k = I\otimes \dots \otimes I \otimes \mathcal{P}_{k-1}^T\otimes \dots \otimes \mathcal{P}_1^T$ for $k > 1$, so we may write 
    \begin{align*}
        E_k = \|\mathcal{A}\times_{i=1}^{k-1} \mathcal{P}_i-\mathcal{A}\times_{i=1}^{k} \mathcal{P}_i\|_F &= \|\mathcal{A}_k B_k - \mathcal{P}_k\mathcal{A}_kB_k\|_F.
    \end{align*}
    We now follow the analogous steps to the proof for Lemma~\ref{lemma_2}, but taking into account 
    the effect of $B_k$, 
    which yields
     \begin{align*}
    E_k^2 &= \|\mathcal{A}_k B_k - Q_k Q_k^T\mathcal{A}_k B_k + Q_k Q_k^T \mathcal{A}_k B_k - \mathcal{A}_k X_k (Y_k^T \mathcal{A}_k X_k)^\dagger Y_k^T \mathcal{A}_kB_k\|_F^2\\
     & = \|Q_{k\perp}Q_{k\perp}^T \mathcal{A}_k B_k +Q_k(Q_k^T - (Y_k^T Q_k)^\dagger Y_k^T)\mathcal{A}_k B_k\|_F^2\\
     & = \| Q_{k\perp}Q_{k\perp}^T\mathcal{A}_k B_k -Q_k(Y_k^T Q_k)^\dagger (Y_k^TQ_{k\perp})Q_{k\perp}^T\mathcal{A}_k B_k \|_F^2\\
        & = {\| Q_{k\perp}Q_{k\perp}^T\mathcal{A}_k B_k\|_F^2 + \| Q_k(Y_k^T Q_k)^\dagger (Y_k^TQ_{k\perp})Q_{k\perp}^T\mathcal{A}_kB_k \|_F^2}\\
         \qquad & \leq {\|Q_{k\perp}^T\mathcal{A}_k B_k\|_F^2 + \| (Y_k^T Q_k)^\dagger\|_2^2 \|Y_k^TQ_{k\perp}\|_2^2 \|Q_{k\perp}^T\mathcal{A}_k B_k \|_F^2}\\
          & \leq \|Q_{k\perp}^T\mathcal{A}_k B_k\|_F^2 \Big( {1 + \| (Y_k^T Q_k)^\dagger\|_2^2 \|Y_k^TQ_{k\perp}\|_2^2 \Big)}\\
         & =\|Q_{k\perp}^T(\mathcal{A}_k + \mathcal{A}_k
B_k - \mathcal{A}_k)\|_F^2 \Big( 1 + \| (Y_k^T Q_k)^\dagger\|_2^2 \|Y_k^TQ_{k\perp}\|_2^2 \Big) \\
&\leq \Big( 
  \|Q_{k\perp}^T \mathcal{A}_k\|_F + \|\mathcal{A}_k -\mathcal{A}_k B_k\|_F \Big)^2 \Big( 1 + \| (Y_k^T Q_k)^\dagger\|_2^2 \|Y_k^TQ_{k\perp}\|_2^2 \Big).
     \end{align*}
     As in the previous result, the thesis follows by taking the square root on both sides of the identity. 
\end{proof}

\begin{remark} \label{remark_HMT_error}
  The term $\|Q_{k\perp}^T \mathcal{A}_k\|_F$ is the approximation error of $HMT$ of the matrix $\mathcal{A}_k$ with sketch matrix $X_k$. Thus, by \eqref{eq:HMT_error_bound} we have 
  \[
  \|Q_{k\perp}^T \mathcal{A}_k\|_F\leq \|\widetilde{\Sigma}_k\|_F\sqrt{1 + \|\widetilde{V}_{k \perp}^{T}X_k\|_2^2\|(\widetilde{V}_k^T X_k)^\dagger\|^2_2}
  \]
where, for any $\hat{r}_k < r_k$, $\widetilde{V}_k$ is the orthogonal matrix with the leading $\hat{r}_k$ right singular vectors of $\mathcal{A}_k$ and $\widetilde{\Sigma}_k$ is the diagonal term in the SVD of $\mathcal{A}_k$ with a $0$ in place of the first $\hat{r}_k$ singular values.
\end{remark}

We now combine these results to state a deterministic 
accuracy bound for MLN. Here, deterministic means that the bound 
is exact as long as the sketchings $X_k, Y_k$ have been fixed. When
these sketches are instead random variables 
with a known distribution, the result yields probabilistic estimates. 

Recall that we denote by $\mlappr_{\varepsilon}(r_1, \ldots, r_d)$ the set of tensors that admit an approximation
with multilinear rank at most $(r_1, \ldots, r_d)$ and an 
approximation error at most $\epsilon$ in the Frobenius norm (see Definition~\ref{def:mlappr}). 

\begin{thm}[Deterministic accuracy bound]\label{thm : deterministic_bound}
 Let $\mathcal{A}\in \mlappr_{\varepsilon}(r_1, \ldots, r_d)$,  let $\hat{\mathcal{A}}$ be the approximant  in \eqref{extended_mlnystrom}, and set 
 \[\tau_k:= \sqrt{1 + \| (Y_k^T Q_k)^\dagger\|_2^2 \|Y_k^TQ_{k\perp}\|_2^2}\quad and \quad \rho_k := \sqrt{1 + \|\widetilde{V}_{k\perp}^{T}X_k\|_2^2\|(\widetilde{V}_k^T X_k)^\dagger\|_2^2},\]
where $\widetilde{V}_k$ is an orthogonal matrix with the first $r_k$ right singular vectors of $\mathcal{A}_k$ and $Q_k = \mathrm{orth}(\mathcal{A}_k X_k)$.
 Then, denoting with $\tau = \max_{k} \tau_k$ and $\rho = \max_{k} \rho_k$, the following bound holds:
 \[
 \|\mathcal{A}-\hat{\mathcal{A}}\|_F \leq \varepsilon \rho ((1+\tau)^d-1).
 \]
 \end{thm}
\begin{proof}
Note that $\tau_k$ and $\rho_k$ 
are defined in terms of $Q_k$ and $\widetilde V_k$, which are uniquely determined
only up to 
right multiplication by appropriate unitary matrices; we start by verifying that $\tau_k$ and $\rho_k$ do not depend on the specific choice
of $Q_k$ and $\widetilde V_k$, and are therefore well-defined.

Since 
$\widetilde V_k^T X_k$ is square and 
$Y_k^T Q_k$ has more rows than columns,
for any unitary matrices $Z, W$
\[
  (Z^T \widetilde V_k^T X_k)^\dagger = (\widetilde V_k^T X_k)^\dagger Z, \qquad 
  (W^T Y_k^T Q_k)^\dagger = 
  (Y_k^T Q_k)^\dagger W.
\]
Hence, thanks to the invariance of the spectral 
norm under unitary transformation, we conclude that 
$\tau_k$ and $\rho_k$ do not
depend on the specific choice of $Q_k$ and 
$\widetilde V_k$, as desired. 

We now prove that the sought inequality holds. 
Thanks to \eqref{sum_formula}, to 
obtain the claim it is sufficient to bound terms of the form $E_k = \|\mathcal{A}\times_{i=1}^{k-1}\mathcal{P}_i-\mathcal{A}\times_{i=1}^k \mathcal{P}_i\|_F$, which then leads to the upper bound 
\begin{equation} \label{eq:sumEk}
  \| \mathcal A - \hat{\mathcal A} \|_F \leq
{\sum_{k=1}^d E_k}.
\end{equation}
We use Lemma \ref{lemma_2} and Remark \ref{remark_HMT_error} to obtain the upper bound 
\[
E_1 \leq \varepsilon_1 \rho_1 \tau_1,
\]
where $\varepsilon_k$ denotes the best possible error of approximation in Frobenius norm of rank $r_k$ of $\mathcal{A}_k$. 
Similarly, we make 
use of Lemma~\ref{lemma_3} for all the remaining modes, which yields for $k = 1,\dots, d-1$
the recurrence relation 
\[
E_{k+1} \leq (\varepsilon_k\rho_k + \|\mathcal{A}-\mathcal{A}\times_{i\leq k} \mathcal{P}_i\|_F)\tau_k \leq (\varepsilon_k\rho_k + \sum_{i\leq k} E_i) \tau_k\leq (\varepsilon \rho + \sum_{i\leq k} E_i) \tau.
\]
In the last inequality, we used $\max_k\varepsilon_k\leq \varepsilon$, 
which holds 
thanks to $\mathcal{A}\in \mlappr_{\varepsilon}(r_1, \ldots, r_d)$.
Then, the $E_k$ for $k = 1, \dots, d$ satisfy the vector inequality
\begin{equation} \label{eq:triangular system}
    \begin{bmatrix}
        \tau^{-1} & 0 & \cdots& \cdots & 0 \\
        -1 & \ddots & \ddots & &\vdots\\
        \vdots & \ddots & \ddots & \ddots& \vdots\\
        \vdots &  & \ddots & \ddots & 0\\
        -1 & \cdots & \cdots & -1 &  \tau^{-1} 
    \end{bmatrix}
    \begin{bmatrix}
    E_1\\
    \vdots\\
     \vdots\\
      \vdots\\
    E_d
    \end{bmatrix}\leq \begin{bmatrix}
     \varepsilon\rho\\
    \vdots\\
     \vdots\\
      \vdots\\
    \varepsilon \rho
    \end{bmatrix}.
\end{equation}
Let $T_k$ be the $k\times k$ principal minor of the lower triangular matrix in \eqref{eq:triangular system}; Both $T_k$
and its inverse are 
lower triangular Toeplitz matrices, and we have the
explicit formula:
\[
 (T_k^{-1})_{ij} = \begin{cases}
     0 & \text{if } i < j \\ 
     \tau & \text{if } i = j \\ 
    \tau \left[ (1 + \tau)^{i-j}  - (1 + \tau)^{i-j-1} \right] & 
      \text{if } i > j
 \end{cases} 
\]
In particular, $T_k^{-1}$ is non-negative since $\tau > 0$, 
so we can left-multiply 
inequality \eqref{eq:triangular system}
by $e_k^T T_k^{-1}$ and obtain an upper 
bound for $E_k$:
\begin{equation}
    E_k \leq e_k^T T_k^{-1} \begin{bmatrix}
     \varepsilon \rho\\
    \vdots\\
    \varepsilon \rho
    \end{bmatrix}.
\end{equation}
Using the explicit expression of the entries in the row vector 
$e_k^T T_k^{-1}$, we finally obtain the upper bound 
$
E_k \leq  \varepsilon \rho \tau(1+\tau )^{k-1}.
$
We conclude using \eqref{eq:sumEk}:
\begin{equation*}
\|\mathcal{A}-\hat{\mathcal{A}}\|_F
\leq \sum_{k=1}^d \varepsilon \rho \tau (1+\tau)^{k-1}
\leq \varepsilon\rho ((1+\tau)^d -1).
\end{equation*}
\end{proof}

\subsection{Accuracy of SMLN} \label{sec:accuracy of SMLN}

We now consider a modification of the proposed approach that improves the stability properties: we replace any pseudoinverse 
$M^\dagger$ appearing in the formulas with its regularized 
counterpart $M^\dagger_\epsilon$, which consists in treating the
singular values below $\epsilon$ in $M$ as zeros
(see the definition of $\epsilon$-pseudoinverse at the 
end of Section~\ref{sec:randomized}).
We refer to such modification as 
stabilized multilinear Nystr\"om (SMLN). In practice, this amounts to replacing the projections $\mathcal P_k$
with 
\[
  \widetilde{\mathcal P}_k :=  \mathcal{A}_k X_k (Y^T_k \widetilde{\mathcal{A}}_k X_k)_\epsilon^\dagger Y_k^T.
\]
This modification is motivated by the following 
observations made in the 
analysis of the matrix case in \cite{nakatsukasa2020fast}:
\begin{enumerate}
    \item For the matrix GN, this change does not substantially change 
      the attainable accuracy;
    \item This modification makes the method reliable in the presence of 
      inexact floating point arithmetic.
\end{enumerate}
This and the next section~\ref{sec:stability} 
investigate if the same results
hold in the tensor case, for the generalization discussed in this paper. 
This section covers the first item (the accuracy), whereas 
the next section \ref{sec:stability} discusses the second item (the 
stability).

Concerning the accuracy, 
we prove that the computed approximant 
$\hat{\mathcal A} = \mathcal A \times_{k = 1}^d \widetilde {\mathcal P}_k$ 
attains, in exact arithmetic, 
an error estimate of the form 

\[
 \| \mathcal A - \hat{\mathcal A} \|_F \leq 
\frac{(1 + \tilde \tau)^d - 1}{\tilde \tau} 
   \max_{k = 1, \ldots, d} 
        \| E_{SGN}^{(k)} \|_F,
\]
where $E_{SGN}^{(k)}$ is the error of the matrix SGN approximation computed in exact arithmetic for the 
      mode-$k$ matricization $\mathcal{A}_k$ with sketchings $X_k$, $Y_k$ and where $\widetilde \tau$ is chosen so that 
$\| \widetilde{\mathcal P}_k \|_2 \leq \widetilde \tau$. 

In relation to the stability result, to obtain a similar estimate, it will be necessary to take into account the floating-point error of approximation. In doing so we will make the simplifying assumption 
that $\mathcal A_k X_k$ and $Y_k^T \mathcal A_k X_k$ are computed
exactly. On one hand, this is obviously unrealistic. On the other hand, any 
sketching low-rank approximation method will use matrices of this form, and 
will thus incur a similar approximation error; in this work, we prefer to focus 
on the error introduced by the algorithmic choice in SMLN that happens after the 
sketching.

The upcoming theorem proves the accuracy result discussed above.

\begin{thm} \label{SMLN_stability}
    Let $\widetilde{\mathcal{P}}_k := \mathcal{A}_k X_k (Y^T_k \widetilde{\mathcal{A}}_k X_k)_\epsilon^\dagger Y_k^T$, where $\widetilde{\mathcal{A}}_k = \mathcal{A}_k+\delta \mathcal{A}_k$ and set 
    \[
    \widetilde{\varepsilon}_k := \|\mathcal{A}-\mathcal{A}\times_k \widetilde{\mathcal{P}}_k\|_F\quad and \quad  \widetilde{\tau}_k := 1 + \|\widetilde{\mathcal{P}}_k\|_2.
    \]
    Then, denoting with $\widetilde\varepsilon:= \max_k\widetilde \varepsilon_k$ and $\widetilde\tau:= \max_k \widetilde\tau_k$, we have
    \[
    \|\mathcal{A}-\mathcal{A}\times_{k=1}^d\widetilde{\mathcal{P}}_k\|_F \leq  \frac{\widetilde\varepsilon}{\widetilde\tau}((1+\widetilde\tau)^d-1).
    \]
\end{thm}

\begin{proof}
Since
\[
\mathcal{A}-\mathcal{A}\times_{k=1}^d \widetilde{\mathcal{P}}_k = \sum_{k=1}^d \mathcal{A}\times_{i=1}^{k-1}\widetilde{\mathcal{P}}_i-\mathcal{A}\times_{i=1}^k \widetilde{\mathcal{P}}_i,
\]
by the subadditivity of the Frobenius norm, we can write
\[
\|\mathcal{A}-\mathcal{A}\times_{k=1}^d \widetilde{\mathcal{P}}_k\|_F \leq \sum_{k=1}^d\|\mathcal{A}\times_{i=1}^{k-1}\widetilde{\mathcal{P}}_i-\mathcal{A}\times_{i=1}^k \widetilde{\mathcal{P}}_i\|_F.
\]
Observe that $E^{(1)}_{SGN}:=\|\mathcal{A}-\mathcal{A}\times_1 \widetilde{\mathcal{P}}_1\|_F$ is exactly the error of the SGN approximant of $\mathcal{A}_1$ with sketches $X_1$ and $Y_1$.
For $k\geq 2$, let $(\mathcal{A}\times_{i=1}^{k-1} \widetilde{\mathcal{P}}_i)_k := \mathcal{A}_k \widetilde{B}_k$, with $\widetilde{B}_k = I\otimes \dots \otimes  I \otimes \widetilde{\mathcal{P}}_{k-1}^T\otimes \dots \otimes \widetilde{\mathcal{P}}_1^T$.
We have
\begin{align*}
\|\mathcal{A}\times_{i=1}^{k-1}\widetilde{\mathcal{P}}_i- \mathcal{A}\times_{i=1}^{k}\widetilde{\mathcal{P}}_i\|_F &=  \|\mathcal{A}_k\widetilde{B}_k-\widetilde{\mathcal{P}}_k\mathcal{A}_k \widetilde{B}_k\|_F = \|(I-\widetilde{\mathcal{P}}_k)\mathcal{A}_k \widetilde{B}_k\|_F\\
&  = \|(I-\widetilde{\mathcal{P}}_k)(\mathcal{A}_k - \mathcal{A}_k + \mathcal{A}_k \widetilde{B}_k)\|_F\\
& \leq \|(I-\widetilde{\mathcal{P}}_k)\mathcal{A}_k\|_F + \|(I-\widetilde{\mathcal{P}}_k) (\mathcal{A}_k -\mathcal{A}_k \widetilde{B}_k)\|_F.
\end{align*}
Again, $E^{(k)}_{SGN} :=\|(I-\widetilde{\mathcal{P}}_k)\mathcal{A}_k\|_F$ is the error of the SGN approximant $\mathcal{A}_k$ with sketches $X_k$ and $Y_k$.

It remains to bound the term $\|(I-\widetilde{\mathcal{P}}_k) (\mathcal{A}_k -\mathcal{A}_k \widetilde{B}_k)\|_F$.
We can write the following chain of inequalities:
\begin{align*}
    \|(I-\widetilde{\mathcal{P}}_k) (\mathcal{A}_k -\mathcal{A}_k \widetilde{B}_k)\|_F &\leq \| I-\widetilde{\mathcal{P}}_k\|_2 \|\mathcal{A}_k -\mathcal{A}_k \widetilde{B}_k\|_F 
    \leq (1+\|\widetilde{\mathcal{P}}_k\|_2) \|\mathcal{A}_k -\mathcal{A}_k \widetilde{B}_k\|_F \\
    &= \widetilde \tau_k\|\mathcal{A}_k -\mathcal{A}_k \widetilde{B}_k\|_F = \widetilde \tau_k \|\mathcal{A}- \mathcal{A}\times_{i=1}^{k-1}\widetilde{\mathcal{P}}_i\|_F \\
    &\leq \widetilde \tau_k \sum_{s=1}^{k-1} \|\mathcal{A}\times_{i=1}^{s-1}\widetilde{\mathcal{P}}_i-\mathcal{A}\times_{i=1}^{s} \widetilde{\mathcal{P}}_i\|_F.
\end{align*}
Summarizing, denoting with 
$E^{(k)}:= \|\mathcal{A} \times_{i=1}^{k-1} \widetilde{\mathcal{P}}_i - \mathcal{A}\times_{i=1}^{k} \widetilde{\mathcal{P}}_i\|_F$, 
we can write the following recurrence relation for an upper bound to the approximation error:
 \[
 \begin{cases}
     E^{(1)} = E^{(1)}_{SGN}\leq \widetilde\varepsilon\\
     E^{(k)} \leq E_{SGN}^{(k)} + \widetilde \tau_k \sum_{i=1}^{k-1}E^{(i)}\leq \widetilde\varepsilon +  \widetilde \tau\sum_{i=1}^{k-1}E^{(i)}
 \end{cases}
\]
The latter system of inequalities is similar to that in \eqref{eq:triangular system} and analogous steps lead to the sought bound.
\end{proof}

The structure of the bound for SMLN is similar to the 
one of MLN. However, in this second case $\widetilde\tau_k$ depends on the norm of $\widetilde{\mathcal{P}}_k$,
which in general can be not negligible
(for instance it grows as $\mathcal{O}(\sqrt{n_k})$ in the Gaussian case \cite{nakatsukasa2020fast}). 

In practice, the proof of Theorem~\ref{SMLN_stability} can be modified 
to obtain a sharper bound. 
We avoided this change in Theorem~\ref{SMLN_stability} for the sake
of clarity, but we give some details here. The key idea is to modify the bound for $\| (I - \widetilde{\mathcal P}_k)(\mathcal{A}_k -\mathcal{A}_k \widetilde{B}_k)\|_F$ in a way that
changes the $\mathcal O(\sqrt{n_k})$ term into a $\mathcal O(\sqrt{r_k})$. 
We may write
\[
\|(I-\widetilde{\mathcal{P}}_k) (\mathcal{A}_k -\mathcal{A}_k \widetilde{B}_k)\|_F \leq \|\mathcal{A}_k -\mathcal{A}_k \widetilde{B}_k\|_F + \|\widetilde{\mathcal{P}}_k(\mathcal{A}_k -\mathcal{A}_k \widetilde{B}_k)\|_F.
\]
The second term in the formula satisfies
 \begin{align*}
     \|\widetilde{\mathcal{P}}_k (\mathcal{A}_k -\mathcal{A}_k \widetilde{B}_k)\|_F = \|\mathcal{A}_k X_k (Y^T_k \mathcal{A}_k X_k)^\dagger_\epsilon Y^T_k (\mathcal{A}_k -\mathcal{A}_k \widetilde{B}_k)\|_F\\
     \leq \|\mathcal{A}_k X_k (Y^T_k \mathcal{A}_k X_k)_\epsilon^\dagger\|_2 \| Y^T_k (\mathcal{A}_k -\mathcal{A}_k \widetilde{B}_k)\|_F.
 \end{align*}
We refer the reader to \cite[Theorem 3.3]{nakatsukasa2020fast} for a 
detailed discussion on how to proceed from here to derive a sharper 
bound, since the remaining steps 
coincide for $d = 2$ or $d > 2$. 
At a high level, we need to check two facts:
the term $\|\mathcal{A}_k X_k (Y^T_k \mathcal{A}_k X_k)_\epsilon^\dagger\|_2$ is $\mathcal{O}(1)$, while $\| Y^T_k (\mathcal{A}_k -\mathcal{A}_k \widetilde{B}_k)\|_F$ is $\mathcal{O}(\sqrt{r_k})   \|\mathcal{A}_k -\mathcal{A}_k \widetilde{B}_k\|_F$. The proof requires Gaussianity of $X_k$, $Y_k$.
However in practical applications
any class of random matrices such that the entries are $\mathcal O(1)$ and a rectangular realization is well-conditioned would work well, including the SRFT and SRHT matrices.

\subsection{Stability of SMLN}
\label{sec:stability}

So far we have analyzed the accuracy of MLN and SMLN without taking into account roundoff errors in floating-point arithmetic. In this section, we address this potential issue.

Regarding the instability of MLN, the situation is similar to that of GN
analyzed in \cite{nakatsukasa2020fast}: stability cannot be established, but the instability is usually benign and one obtains satisfactory results. In this section, we establish the numerical stability of SMLN.

Consider the stabilized MLN approximation 
$(\mathcal{A}\times_{k=1}^d Y_k^T) \times_{k=1}^d \mathcal{A}_k X_k (Y_k^T \mathcal{A}_k X_k)_\epsilon^{\dagger}$ obtained in finite precision-arithmetic. We assume to compute each Tucker factor
$\widetilde{W}_k :=\mathcal{A}_k X_k (Y_k^T \mathcal{A}_k X_k)_\epsilon^{\dagger}$ and $\mathcal{C} =\mathcal{A}\times_{k=1}^d Y_k^T$ separately. Then, we compare the floating representation with the original tensor $\mathcal{A}: \|\mathcal{A}-\mathcal{C}\times_{k=1}^d \fl(\widetilde{W}_k)\|_F$, where we used the 
notation $\fl(\cdot)$ to denote the outcome of an arithmetic operation in 
floating point arithmetic. 

As we have anticipated in Section~\ref{sec:accuracy of SMLN}, and in line with the analysis in \cite{nakatsukasa2020fast}, 
we will assume that each row of $\widetilde{W}_k$ is computed by a backward stable underdetermined linear solver and that $\mathcal A_k X_k$ and $Y_k^T \mathcal A_k X_k$ are computed exactly.

We also borrow some notation from \cite{nakatsukasa2020fast}: we use $\OO(1)$ to suppress terms involving dimensions of the problem or the ranks (like $n_k, r_k$),  but not $1/\epsilon, \sigma_r^{-1}(\mathcal{A}_k)$. We use $\epsilon_{*}$ to denote either a tensor, a matrix, or a scalar such that $\|\epsilon_{*}\|_F = \OO(u \|A\|_F)$. The precise value of $\epsilon_{*}$ may change from appearance to appearance. See \cite{nakatsukasa2020fast} and 
\cite{common_practice} for the motivation behind this notation, which is standard practice in stability analysis.

Let us denote with $[M]_i$ the $i$th row of $M$. The following
proofs are heavily based on the results from \cite{nakatsukasa2020fast}
in the matrix case, for which we give precise references. 

\begin{lemma} \label{lemma: floating MP-inverse}

Let $\mathcal{A}$, $X_k$, $Y_k$ such that $X_k,Y_k$ are Gaussian, $\mathcal{A}_k X_k$ is full column-rank and $Y_k^T\mathcal{A}_k X_k$ is tall. Suppose also that $\epsilon = \OO(u \|A\|_F)$ and that each row of $\mathcal{A}_k X_k (Y_k^T\mathcal{A}_k X_k)_\epsilon^\dagger$ is computed by a backward stable underdetermined linear solver. Then with an exponentially high probability
\begin{equation}
    \|\fl(\mathcal{A}_k X_k (Y_k^T\mathcal{A}_k X_k)_\epsilon^\dagger)\|_F \sim \OO(1).
\end{equation}
\end{lemma}
\begin{proof}
We can assume, without loss of generality, to perform a preliminary scaling to have 
$\|\mathcal{A}_k\|_F = 1$.
Following the proof of \cite[Theorem 4.1]{nakatsukasa2020fast} we know that 
\begin{align} \label{eq:square solver}
\|[\fl(\mathcal{A}_k X_k (Y_k^T\mathcal{A}_k X_k)_\epsilon^\dagger)]_i\|_2 &= \|[\mathcal{A}_k X_k + \epsilon_*]_i(Y_k^T\mathcal{A}_k X_k + \epsilon_*)^\dagger_\epsilon \|_2 \\
&\leq \|[\mathcal{A}_k X_k]_i(Y_k^T\mathcal{A}_k X_k +\epsilon_*)^\dagger_\epsilon\|_2 + \OO(1).
\end{align}
Let $U\Sigma V^T$ be the SVD of $\mathcal{A}_k X_k$; we have
\begin{align*}
\|\mathcal{A}_k X_k(Y_k^T\mathcal{A}_k X_k +\epsilon_*)^\dagger_\epsilon\|_2 &= \|\Sigma V^T(Y_k^T\mathcal{A}_k X_k +\epsilon_*)^\dagger_\epsilon\|_2 \\
& = \|(Y_k^T U)^\dagger (Y_k^T\mathcal{A}_k X_k)(Y_k^T\mathcal{A}_k X_k +\epsilon_*)^\dagger_\epsilon\|_2.    
\end{align*}
Let us denote by $\Xi_i$ the error matrix in the above 
expression: 
we fix $(Y_k^T\mathcal{A}_k X_k +\epsilon_*)^\dagger_\epsilon = (Y_k^T\mathcal{A}_k X_k +\Xi_i)^\dagger_\epsilon$, and we have $\| \Xi_i \|_2 \sim \OO(u)$, 
Then, 
\begin{align*}
\|\mathcal{A}_k X_k(Y_k^T\mathcal{A}_k X_k +\epsilon_*)^\dagger_\epsilon\|_2 &\leq 
\|(Y_k^T U)^\dagger\|_2 \|(Y_k^T\mathcal{A}_k X_k+\Xi + \epsilon_*)(Y_k^T\mathcal{A}_k X_k +\Xi)^\dagger_\epsilon\|_2\\
& \leq \|(Y_k^T U)^\dagger\|_2 \|(Y_k^T\mathcal{A}_k X_k+\Xi_i)(Y_k^T\mathcal{A}_k X_k +\Xi_i)^\dagger_\epsilon\|_2 \\
&+ \|(Y_k^T U)^\dagger\|_2\|\epsilon_* (Y_k^T\mathcal{A}_k X_k +\epsilon_*)^\dagger_\epsilon\|_2 \sim  \OO(1).
\end{align*}
In the last equality we used that $\|(Y_k^T U)^\dagger\|_2\sim \OO(1)$, which follows from the fact that $Y_k^T U$ is tall-Gaussian, hence well-conditioned, that $\|(Y_k^T\mathcal{A}_k X_k+\Xi_i)(Y_k^T\mathcal{A}_k X_k +\Xi_i)^\dagger_\epsilon\|_2 = 1$ and that $\|\epsilon_* (Y_k^T\mathcal{A}_k X_k +\Xi_i)^\dagger_\epsilon\|_2\leq \|\epsilon_*\|_2/\epsilon \sim \OO(1)$.

Then, the rows of the computed matrix satisfy $\|[\fl(\mathcal{A}_k X_k (Y_k^T\mathcal{A}_k X_k)_\epsilon^\dagger)]_i\|_2 = \OO(1)$. As a consequence, also the Frobenius norm of the computed matrix is $\OO(1)$.
\end{proof}

We are now ready to prove the main stability result.

\begin{thm}
 \label{theorem: stability}
   Let the assumptions in Lemma \ref{lemma: floating MP-inverse} be satisfied and suppose to form $\mathcal{C}\times_{k=1}^d \widetilde{W}_k = (\mathcal{A}\times_{k=1}^d Y_k^T)\times_{k=1}^d \mathcal{A}_k X_k (Y_k^T\mathcal{A}_k X_k)_\epsilon^\dagger$ by first computing the core tensor $\mathcal{C}$ and the matrices $\widetilde{W}_k$ and then performing the mode products between $\mathcal{C}$ and the $\widetilde{W}_k$. 
   Set
  \[
   \widetilde{\varepsilon}_{k}:= \|\mathcal{A}- \mathcal{A}\times_k 
   \fl(\widetilde{W}_k) Y_k^T) \|_F \qquad and \qquad \widetilde{\tau}_k:= 1 + \|\fl(\widetilde{W}_k)\|_2\|Y_k^T\|_2.
  \]
  Then, denoting with $\widetilde\varepsilon:= \max_k\widetilde \varepsilon_k$ and $\widetilde\tau:= \max_k \widetilde\tau_k$, we have
    \[
    \|\mathcal{A}-\mathcal{C}\times_{k=1}^d \fl(\widetilde{W}_k)\|_F\leq  \frac{\widetilde\varepsilon}{\widetilde\tau}((1+\widetilde\tau)^d-1).
    \]

\end{thm}
\begin{proof}
The error of the approximation, in view of 
the subadditivity of the Frobenius norm satisfies
\begin{align*}
\|\mathcal{A}-\mathcal{C}\times_{k=1}^d \fl(\widetilde{W}_k)\|_F&=
\| \mathcal{A}-\mathcal{A}\times_{k=1}^d \fl(\widetilde{W}_k)Y_k^T\|_F\\ &\leq \sum_{k=1}^d \|\mathcal{A}\times_{i=1}^{k-1}\fl(\widetilde{W}_i)Y_i^T- \mathcal{A}\times_{i=1}^k \fl(\widetilde{W}_i)Y_i^T\|_F.
\end{align*}
The Frobenius norm is invariant under matricization, so it suffices 
to bound terms of the form $\|(I-\fl(\widetilde{W}_k)Y_k^T)\mathcal{A}_kB_k\|_F$ for $k=1,\dots, d$, where the matrices $B_k$ are given by
$B_k := (I\otimes\dots \otimes I\otimes \fl(\widetilde{W}_{k-1})Y_{k-1}^T\otimes \dots \otimes \fl(\widetilde{W}_{1})Y_{1}^T)^T$.
Then, 
\begin{align*}
&\|(I-\fl(\widetilde{W}_k)Y_k^T)\mathcal{A}_k B_k\|_F = \|(I-\fl(\widetilde{W}_k)Y_k^T)(\mathcal{A}_k + \mathcal{A}_k B_k - \mathcal{A}_k) \|_F \\
&\qquad \leq \|(I-\fl(\widetilde{W}_k)Y_k^T)\mathcal{A}_k\|_F + \|(I-\fl(\widetilde{W}_k)Y_k^T)(\mathcal{A}_k- \mathcal{A}_k B_k) \|_F\\
&\qquad \leq \|\mathcal{A}- \mathcal{A}\times_k \fl(\widetilde{W}_k)Y_k^T \|_F+ (1+\|\fl(\widetilde{W}_k)\|_2\|Y_k^T\|_2)\|\mathcal{A}_k - \mathcal{A}_k B_k\|_F.\\
&\qquad \leq \|\mathcal{A}_k- \fl(\mathcal{A}_k X_k (Y^T\mathcal{A}_k X_k)^\dagger_\epsilon) Y_k^T\mathcal{A}_k\|_F + \widetilde{\tau}_k\|\mathcal{A}_k - \mathcal{A}_k B_k\|_F\\
&\qquad \leq \widetilde{\varepsilon}_k +\widetilde{\tau_k}\sum_{k=1}^d \|\mathcal{A}\times_{i=1}^{k-1}\fl(\widetilde{W}_i)Y_i^T- \mathcal{A}\times_{i=1}^k \fl(\widetilde{W}_i)Y_i^T\|_F.
\end{align*}

We obtained a recurrence inequality with the same structure encountered in the proof of Theorem \ref{SMLN_stability}. Thus, similar steps lead to the sought inequality.
\end{proof}
We emphasize that the parameter $\widetilde{\varepsilon}_k=\|\mathcal{A}- \mathcal{A}\times_k \fl(\widetilde{W}_k)Y_k^T \|_F$ is the error of approximation of the SGN method in floating point arithmetic performed on $\mathcal{A}_k$ with sketch matrices $X_k, Y_k$ and has been extensively analyzed in \cite[Section 4.1]{nakatsukasa2020fast}. The parameter $\widetilde{\tau}_k$ is bounded thanks to Lemma~\ref{lemma: floating MP-inverse}. To facilitate the understanding of the proof, we set $\widetilde{\tau}_k$ equal to $1 + \|\fl(\widetilde{Q}_k)\|_2\|Y_k^T\|_2$; however a sharper estimate may be obtained following the strategies in \cite[Section 4.1]{nakatsukasa2020fast}.
Again, the Gaussian hypothesis is not strictly necessary, and any class of random matrices such that the entries are $\mathcal O(1)$ and a rectangular realization is well-conditioned would work well.

\section{Sketch selection for MLN}

An important aspect of MLN is the choice of the sketch matrices. 
Indeed, in the deterministic bound for MLN \eqref{thm : deterministic_bound} conditioning terms that depend on the probability distribution of the sketching appear.

If the tensor $\mathcal{A}$ is not structured, several options are available: Gaussian, SRHT, SRFT, DCT, just to name a few. 
We recommend the recent survey \cite{randomization_advantajes} for an excellent overview of randomized algorithms in numerical linear algebra.

When $\mathcal{A}$ is structured instead, specific DRM should be used.
Of particular interest is the case where the tensor $\mathcal{A}$ is given in Tucker format, that is $\mathcal{A}= \mathcal{C}\times_{i=1}^d U_i$, where $\mathcal{C}\in\mathbb{R}^{k_1\times \dots \times k_d}$ is a small tensor and the $U_i$ have size $n_i\times k_i$ and we would compress it to obtain a Tucker tensor of smaller dimensions.

In this case, the flattening along the $k$th index is given by $U_k\mathcal{C}_k \Ukk^T$ and an appropriate right DRM could exploit the Kronecker structure of $\Ukk$ to accelerate the sketching procedure.
For example, one may choose as sketching matrix a Kronecker product of small independent DRMs $\Omega_i\in \mathbb{R}^{n_i\times r_i}$, $r_i$, like the ones described for the unstructured case, in order to compute $\Ukk \Omegakk$ in a very cheap way.

This strategy has the drawback that to obtain a satisfactory approximation of a matrix with rank $r$ with high probability, each $\Omega_i$ should have at least $r$ columns, for a total of $r^d$ columns. Another possibility, which is what we suggest, is to use a random selection of columns from the Kronecker product of the $\Omega_i$. 
These types of sketchings have been extensively analyzed in the literature; in particular, in \cite{Kolda} KFJLT are proposed. KFJLTs drastically reduce the embedding cost to an exponential factor of the standard fast Johnson-Lindenstrauss transform (FJLT)’s cost when applied to vectors with Kronecker structure and, above all, the computational gain comes with only a small price in embedding power.
A related strategy is selecting the right DRMs with Khatri-Rao structure, as done in \cite{R-ST-HOSVD} in a similar context.

Our results allow us to analyze the accuracy attained from the approximation 
methods obtained by a specific choice of sketchings $Y_k, X_k$ 
in terms of the singular values and the condition number 
of the matrices $Y_k^T Q_k$, with an orthogonal $Q_k$ 
(see Theorem~\ref{thm : deterministic_bound}). 

\section{Experiments}

The aim of this section is to illustrate the performances of MLN and SMLN and to show that the theoretical bounds provided in Theorem $\ref{thm : deterministic_bound}$ are sharp.
The code used for the numerical experiments is available at \url{https://github.com/alb95/MLN}.

\subsection*{Implementation}

In our implementation of (S)MLN the core tensor $\mathcal{A}\times_{k=1}^d Y_k^T$ and the matrices $\mathcal{A}_k X_k (Y_k^T\mathcal{A}_k X_k)^{\dagger}$ are computed separately. Regarding the implementation of the pseudo-inverse and of the $\epsilon$-truncated pseudo-inverse, see \cite[Section 5]{nakatsukasa2020fast}.

\subsection*{Numerical illustration}

In the experiments, we will compare the performances of MLN and SMLN methods with other popular methods for the low-rank tensor approximation in Tucker format. In particular, we will compare the method with the truncated higher-order singular value decomposition (HOSVD), the randomized HOSVD (RHOSVD), and the randomized sequential truncated HOSVD (RSTHOSVD). 

All numerical experiments were performed in MATLAB version 2023b on a laptop with 16GB of system memory and 
the tensor operations (mode-$k$ products, unfoldings, and 
the computation of the high-order SVD) have been performed 
by means of the Tensor Toolbox for MATLAB v3.5 \cite{tensor_toolbox} and the \texttt{tensorprod} function in MATLAB.

We recall that in the HOSVD method the SVD of each of the $d$ matricizations of the tensor $\mathcal{A}$ is computed, i.e. $U_k S_k V_k^T = \mathcal{A}_k$, and then, set $r = (r_1, \dots, r_d)$ as the desired multilinear rank of the approximant, the tensor $\hat{\mathcal{A}} = (\mathcal{A}\times_{k=1}^d U_k^{(r_k)T})\times_{k=1}^d U_k^{(r_k)}$ is formed; where $U_k^{(r_k)}$ are the first $r_k$ columns of $U_k$.

The RHOSVD is similar: once the multilinear rank $r$ of the approximant is fixed, we draw $d$ random sketchings $X_k$ of size $\prod_{j\neq k} n_j\times r_k$ and then we compute an economy-size SVD of the matrices $U_k S_k V_k^T = \mathcal{A}_kX_k$. The approximant is given by $\hat{\mathcal{A}} = (\mathcal{A}\times_{k=1}^d U_k^T)\times_{k=1}^d U_k$.

The RSTHOSVD differs from previous methods as it truncates the tensor while processing each mode. More specifically,
if the target multilinear rank $r$ of the approximant is chosen, $d$ random sketchings $X_k$ of size $\left(\prod_{j< k}r_j\right) \left(\prod_{j> k} n_j\right)\times r_k$ are drawn. Then, an economy-size SVD of the matrices $U_k S_k V_k^T = \mathcal{B}^{(k)}_k X_k$ are sequentially computed, where $\mathcal{B}^{(k)}$ denotes the partially truncated core tensor $\mathcal{B}^{(k)} = \mathcal{A}\times_{j=1}^k U_j^T$.
In this case the approximant is given by $\hat{\mathcal{A}} = (\mathcal{A}\times_{k=1}^d U_k^T)\times_{k=1}^d U_k$.

Even if the truncated HOSVD is often prohibitively expensive, it is an important benchmark as it provides an almost optimal Tucker approximant, i.e. if $\hat{\mathcal{A}}$ is the truncated HOSVD and $\mathcal{A}_*$ is the optimal solution to the best low multilinear rank approximation problem, then 
\[
\|\mathcal{A}-\hat{\mathcal{A}}\|_F\leq \sqrt{d}\|\mathcal{A}-\mathcal{A}_*\|_F.
\]

In the majority of the experiment, to describe different decays, we construct the tensors by fixing their CP-decomposition. That is, we fix a sequence of $n$ decreasing positive numbers $\sigma_i$, we generate $d$ matrices $Q_i\in \mathbb{R}^{n\times n}$, where $Q_i$ is a random orthogonal matrix (Q-factor in the QR factorization of a square Gaussian matrix) and we set $\mathcal{T} = \mathcal{S}\times_{i=1}^d Q_i$, where $\mathcal{S}$ is the superdiagonal tensor with the $\sigma_i$ in the superdiagonal.
Notice that in this way the singular values of each matricization are the $\sigma_i$.

Unless explicitly specified, the sketch matrices are assumed to be Gaussian.

In the first experiment, Figure \ref{fig:Stabilization comparison}, we compare the performances of MLN and SMLN. To do so, we test the algorithms on a numerically low-rank tensor $\mathcal{T}$ of size $70\times 70\times 70$ with exponential decay in the $\sigma_i$ of rate $0.1$ (i.e., $\sigma_i = 0.1^i$).

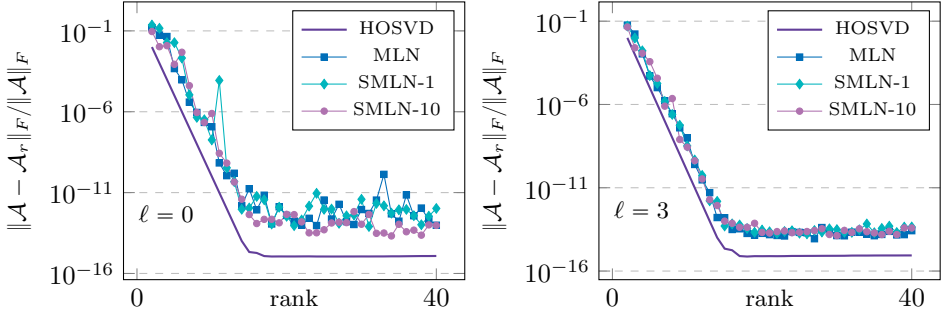
\begin{figure}
\centering\begin{tikzpicture}
\begin{semilogyaxis}[
    x label style={at={(axis description cs:0.5,0)},anchor=north},
    xlabel={\small{rank}},
    ylabel={\small{$\|\mathcal{A}-\mathcal{A}_r\|_F/\|\mathcal{A}\|_F$}},
    legend pos=north west,
    xtick = {0,40},
    ymajorgrids=true,
    grid style=dashed,
    width=.47\linewidth,
    legend pos = north east,
    mark size = 1.2pt,
    restrict x to domain=0:40
]

    \addplot [color = RoyalPurple, style = thick] table [col sep=comma, x index = 0, y index = 1] {E_stabilization_no_oversample.txt};
  \addplot [color = NavyBlue, mark = square*] table [col sep=comma, x index = 0, y index = 2] {E_stabilization_no_oversample.txt};
  \addplot [color = Aquamarine, mark = diamond*, mark size= 1.8pt] table [col sep=comma, x index = 0, y index = 3] {E_stabilization_no_oversample.txt};
  \addplot [color = Orchid, mark = *] table [col sep=comma, x index = 0, y index = 4] {E_stabilization_no_oversample.txt};
\node[anchor=west, rotate=0] at (axis cs:-1,5e-13) {\small{$ \ell=0$}};
  \legend{\scriptsize{HOSVD}, \scriptsize{MLN}, \scriptsize{SMLN-1}, \scriptsize{SMLN-10}}

\end{semilogyaxis}
\end{tikzpicture}~\begin{tikzpicture}
\begin{semilogyaxis}[
    x label style={at={(axis description cs:0.5,0)},anchor=north},
    xlabel={\small{rank}},
    ylabel={\small{$\|\mathcal{A}-\mathcal{A}_r\|_F/\|\mathcal{A}\|_F$}},
    legend pos=north west,
    xtick = {0,40},
    ymajorgrids=true,
    grid style=dashed,
    width=.47\linewidth,
    legend pos = north east,
    mark size = 1.2pt,
    restrict x to domain=0:40
]

     \addplot [color = RoyalPurple, style = thick] table [col sep=comma, x index = 0, y index = 1] {E_stabilization_oversample.txt};
  \addplot [color = NavyBlue, mark = square*] table [col sep=comma, x index = 0, y index = 2] {E_stabilization_oversample.txt};
  \addplot [color = Aquamarine, mark = diamond*, mark size= 1.8pt] table [col sep=comma, x index = 0, y index = 3] {E_stabilization_oversample.txt};
  \addplot [color = Orchid, mark = *] table [col sep=comma, x index = 0, y index = 4] {E_stabilization_oversample.txt};
\node[anchor=west, rotate=0] at (axis cs:-1,5e-13) {\small{$ \ell=3$}};
  \legend{\scriptsize{HOSVD}, \scriptsize{MLN}, \scriptsize{SMLN-1}, \scriptsize{SMLN-10}}
\end{semilogyaxis}
\end{tikzpicture}
\caption{Accuracy of multilinear Nystr\"om (MLN) and stabilized multilinear Nystr\"om with $\epsilon = u\|A\|_F$ (SMLN-1) and $\epsilon = 10u\|A\|_F$ (SMLN-10).}
    \label{fig:Stabilization comparison}
\end{figure}

The tests show that when there is no oversampling SMLN performs better than MLN; but both methods produce unsatisfactory results. Instead, even with a small oversampling ($\ell = 3$), both methods show significant improvement and in practice result equivalent. This confirms what was observed in \cite{nakatsukasa2020fast} for GN: stability cannot be established for plain MLN, but oversampling makes its instability benign, and one usually obtains satisfactory results.

Given that both theoretical and experimental results support oversampling's fundamental impact, we recommend its consistent implementation; moreover, due to the equivalence of results achieved with oversampling by MLN and SMLN, we will not include SMLN in the experiments below.

We notice that oversampling, other than to stabilize MLN, serves to improve its accuracy. Hence, we conduct experiments to determine the optimal choice for the oversampling parameter $\ell$. 

The results are shown in Figure \ref{fig:Oversampling comparison}. As observed for GN in \cite{nakatsukasa2020fast}, MLN with fixed $\ell$ gets further from optimal as $r$ increases, whereas choosing $\ell=cr$ avoids this issue. In particular, the examples show that choosing 
$c = \frac 12$ and therefore $\ell = \frac r2$ yields a 
robust implementation in all cases.

\begin{figure}
\centering\begin{tikzpicture}
\begin{semilogyaxis}[
    x label style={at={(axis description cs:0.5,0)},anchor=north},
    xlabel={\small{rank}},
    ylabel={\small{$\|\mathcal{A}-\mathcal{A}_r\|_F/\|\mathcal{A}\|_F$}},
    legend pos=north east,
    legend style={fill=none,draw=white},
    xtick = {0,90},
    ytick = {1e-14, 1e-6,1e2},
    ymajorgrids=true,
    grid style=dashed,
    width=.47\linewidth,
    mark size = 1.2pt,
    restrict x to domain=0:90
]
  \addplot [color = Green, mark = diamond*, mark size= 2pt] table [col sep=comma, x index = 0, y index = 2] {E_oversampling_exponential.txt};
  \addplot [color = CadetBlue, mark = square*,mark size= 1.3pt] table [col sep=comma, x index = 0, y index = 3] {E_oversampling_exponential.txt};
  \addplot [color = NavyBlue, mark = pentagon*, mark size= 1.8pt] table [col sep=comma, x index = 0, y index = 4] {E_oversampling_exponential.txt};
  \addplot [color = SkyBlue, mark = *, mark size= 1.6pt] table [col sep=comma, x index = 0, y index = 5] {E_oversampling_exponential.txt};
\addplot [color = RoyalPurple, style = thick] table [col sep=comma, x index = 0, y index = 1] {E_oversampling_exponential.txt};
\node[anchor=west, rotate=0] at (axis cs:-1,3e-14) {\small{$ \sigma_i = 0.7^i$}};
\node[anchor=west, rotate=-34] at (axis cs:33,8e-7) {\scriptsize{HOSVD}};
  \legend{\scriptsize{$\ell=0$}, \scriptsize{$\ell=2$}, \scriptsize{$\ell=r/2$}, \scriptsize{$\ell=r$}}
\end{semilogyaxis}
\end{tikzpicture}~\begin{tikzpicture}
\begin{semilogyaxis}[
    x label style={at={(axis description cs:0.5,0)},anchor=north},
    xlabel={\small{rank}},
    ylabel={\small{$\|\mathcal{A}-\mathcal{A}_r\|_F/\|\mathcal{A}\|_F$}},
    legend pos=south east,
    xtick = {0,90},
    ytick = {1e-1,1e2,1e4},
    ymajorgrids=true,
    grid style=dashed,
    width=.47\linewidth,
    mark size = 1.2pt,
    restrict x to domain=0:90
]
    \addplot [color = Green, mark = diamond*, mark size= 2pt] table [col sep=comma, x index = 0, y index = 2] {E_oversampling_linear.txt};
    \addplot [color = CadetBlue, mark = square*,mark size= 1.3pt] table [col sep=comma, x index = 0, y index = 3] {E_oversampling_linear.txt};
    \addplot [color = NavyBlue, mark = pentagon*, mark size= 1.8pt] table [col sep=comma, x index = 0, y index = 4] {E_oversampling_linear.txt};
    \addplot [color = SkyBlue, mark = *, mark size= 1.6pt] table [col sep=comma, x index = 0, y index = 5] {E_oversampling_linear.txt};
    \addplot [color = RoyalPurple, style = thick] table [col sep=comma, x index = 0, y index = 1] {E_oversampling_linear.txt};
    \node[anchor=west, rotate=0] at (axis cs:-1,3e-2) {\small{$ \sigma_i=1/i$}};
\end{semilogyaxis}
\end{tikzpicture}
\caption{Performance comparison of MLN with varying values of the oversampling parameter $\ell$ on two tensors of size $100\times 100\times 100$: one with $\sigma_i = 0.7^i$ (left) and another with $\sigma_i = 1/i$ (right)." }
    \label{fig:Oversampling comparison}
\end{figure}
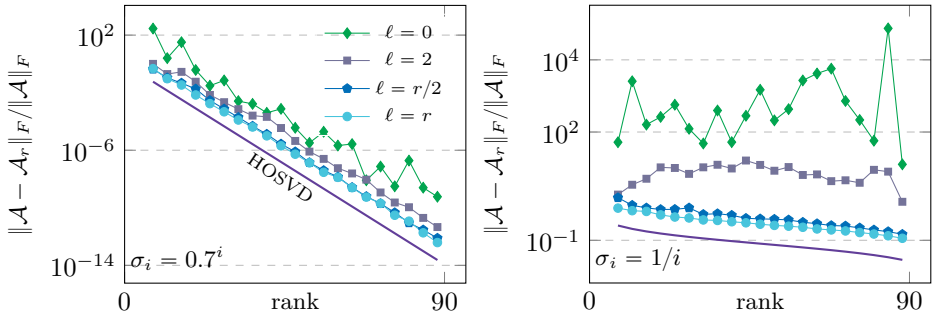

In the upcoming experiment, Figure \ref{fig: various decay comparison}, we compare the performances of MLN (with $\ell=r/2$), RHOSVD, and HOSVD. We test the algorithms on 4 tensors of size $100\times 100\times 100$ with different decays: linear ($\sigma_i = 1/i)$, quadratic ($\sigma_i= 1/i^2)$, cubic ($\sigma_i = 1/i^3$) and exponential ($\sigma_i = 0.5^i$). The plots show that up to a small constant, the accuracy of MLN and RHOSVD are the same and that both achieve near-optimal accuracy.

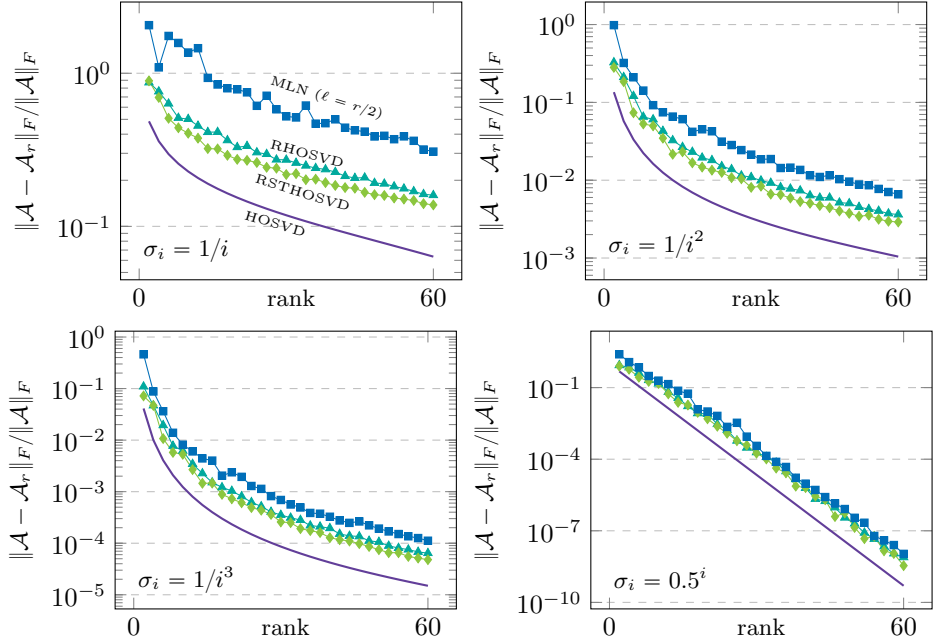
\begin{figure}
\centering\begin{tikzpicture}
\begin{semilogyaxis}[
    x label style={at={(axis description cs:0.5,0)},anchor=north},
    xlabel={\small{rank}},
    ylabel={\small{$\|\mathcal{A}-\mathcal{A}_r\|_F/\|\mathcal{A}\|_F$}},
    xtick = {0,60},
    legend pos=north west,
    ymajorgrids=true,
    grid style=dashed,
    width=.47\linewidth,
    mark size = 1.4pt,
    restrict x to domain=0:60
]

  \addplot [color = RoyalPurple, style = thick] table [col sep=comma, x index = 0, y index = 1] {E_linear.txt};
  \addplot [color = Emerald, mark = triangle*, mark size= 1.8pt] table [col sep=comma, x index = 0, y index = 2] {E_linear.txt};
  \addplot [color = LimeGreen, mark = diamond*, mark size= 1.8pt] table [col sep=comma, x index = 0, y index = 3] {E_linear.txt};
  \addplot [color = NavyBlue, mark = square*] table [col sep=comma, x index = 0, y index = 4] {E_linear.txt};
\node[anchor=west, rotate=343] at (axis cs:25,9.3e-1) {\tiny{MLN ($\ell=r/2$)}};
\node[anchor=west, rotate=345] at (axis cs:25,3.6e-1) {\tiny{RHOSVD}};
\node[anchor=west, rotate=343] at (axis cs:22,2.15e-1) {\tiny{RSTHOSVD}};
\node[anchor=west, rotate=343] at (axis cs:19.8,1.25e-1) {\tiny{HOSVD}};
\node[anchor=west, rotate=0] at (axis cs:-1,7e-2) {\small{$\sigma_i = 1/i$}};
    
\end{semilogyaxis}
\end{tikzpicture}~\begin{tikzpicture}
\begin{semilogyaxis}[
    x label style={at={(axis description cs:0.5,0)},anchor=north},
    xlabel={\small{rank}},
    ylabel={\small{$\|\mathcal{A}-\mathcal{A}_r\|_F/\|\mathcal{A}\|_F$}},
    xtick = {0,60},
    legend pos=north west,
    ymajorgrids=true,
    grid style=dashed,
    width=.47\linewidth,
    mark size = 1.4pt,
    restrict x to domain=0:60
]

     \addplot [color = RoyalPurple, style = thick] table [col sep=comma, x index = 0, y index = 1] {E_quadratical.txt};
   \addplot [color = Emerald, mark = triangle*, mark size= 1.8pt] table [col sep=comma, x index = 0, y index = 2] {E_quadratical.txt};
  \addplot [color = LimeGreen, mark = diamond*, mark size= 1.8pt] table [col sep=comma, x index = 0, y index = 3] {E_quadratical.txt};
  \addplot [color = NavyBlue, mark = square*] table [col sep=comma, x index = 0, y index = 4] {E_quadratical.txt};
\node[anchor=west, rotate=0] at (axis cs:-1,1.5e-3) {\small{$\sigma_i = 1/i^2$}};
  
\end{semilogyaxis}
\end{tikzpicture} \\
\begin{tikzpicture}
\begin{semilogyaxis}[
    x label style={at={(axis description cs:0.5,0)},anchor=north},
    xlabel={\small{rank}},
    xtick = {0,60},
    ylabel={\small{$\|\mathcal{A}-\mathcal{A}_r\|_F/\|\mathcal{A}\|_F$}},
    legend pos=north west,
    ymajorgrids=true,
    grid style=dashed,
    width=.47\linewidth,
    mark size = 1.4pt,
    restrict x to domain=0:60
]
    \addplot [color = RoyalPurple, style = thick] table [col sep=comma, x index = 0, y index = 1] {E_cubic.txt};
    \addplot [color = Emerald, mark = triangle*, mark size= 1.8pt] table [col sep=comma, x index = 0, y index = 2] {E_cubic.txt};
  \addplot [color = LimeGreen, mark = diamond*, mark size= 1.8pt] table [col sep=comma, x index = 0, y index = 3] {E_cubic.txt};
  \addplot [color = NavyBlue, mark = square*] table [col sep=comma, x index = 0, y index = 4] {E_cubic.txt};
  \node[anchor=west, rotate=0] at (axis cs:-1,2e-5) {\small{$\sigma_i = 1/i^3$}};
\end{semilogyaxis}
\end{tikzpicture}~\begin{tikzpicture}
\begin{semilogyaxis}[
    x label style={at={(axis description cs:0.5,0)},anchor=north},
    xlabel={\small{rank}},
    ylabel={\small{$\|\mathcal{A}-\mathcal{A}_r\|_F/\|\mathcal{A}\|_F$}},
    xtick = {0,60},
    ymajorgrids=true,
    grid style=dashed,
    width=.47\linewidth,
    legend pos = north east,
    mark size = 1.4pt,
    restrict x to domain=0:60
]

     \addplot [color = RoyalPurple, style = thick] table [col sep=comma, x index = 0, y index = 1] {E_exponential.txt};
  \addplot [color = Emerald, mark = triangle*, mark size= 1.8pt] table [col sep=comma, x index = 0, y index = 2] {E_exponential.txt};
  \addplot [color = LimeGreen, mark = diamond*, mark size= 1.8pt] table [col sep=comma, x index = 0, y index = 3] {E_exponential.txt};
  \addplot [color = NavyBlue, mark = square*] table [col sep=comma, x index = 0, y index = 4] {E_exponential.txt};
  \node[anchor=west, rotate=0] at (axis cs:-1,1e-9) {\small{$\sigma_i = 0.5^i$}};
    
\end{semilogyaxis}
\end{tikzpicture}
 \caption{Comparison of MLN with oversampling parameter $\ell=r/2$, HOSVD, RHOSVD, and RSTHOSVD on tensors with different decays.}
 \label{fig: various decay comparison}
\end{figure}
The same occurs when we test the algorithms on 3-dimensional and 4-dimensional Hilbert tensors, see Figure \ref{fig:Hilbert tensors}.
\begin{figure}
\centering\begin{tikzpicture}
\begin{semilogyaxis}[
    title={\small{3D Hilbert tensor}},
    x label style={at={(axis description cs:0.5,0)},anchor=north},
    xlabel={\small{rank}},
    ylabel={\small{$\|\mathcal{H}-\mathcal{H}_r\|_F/\|\mathcal{H}\|_F$}},
    xtick = {0, 40},
    ytick={1, 1e-10},
    yticklabels={\small{$1$}, \small{$10^{-10}$}},
    legend pos=north west,
    ymajorgrids=true,
    grid style=dashed,
    width=.47\linewidth,
    legend pos = north east,
    mark size = 1.2pt,
    restrict x to domain=0:40
]

 \addplot [color = RoyalPurple, style = thick] table [col sep=comma, x index = 0, y index = 1] {E_3d_hilbert.txt};
 \addplot [color = Emerald, mark = triangle*, mark size= 1.8pt] table [col sep=comma, x index = 0, y index = 2] {E_3d_hilbert.txt};
  \addplot [color = LimeGreen, mark = diamond*, mark size= 1.8pt] table [col sep=comma, x index = 0, y index = 3] {E_3d_hilbert.txt};
  \addplot [color = NavyBlue, mark = square*] table [col sep=comma, x index = 0, y index = 4] {E_3d_hilbert.txt};
    \legend{\scriptsize{HOSVD}, \scriptsize{RHOSVD},\scriptsize{RSTHOSVD},\scriptsize{MLN}}
    
\end{semilogyaxis}
\end{tikzpicture}~\begin{tikzpicture}
\begin{semilogyaxis}[
    title={\small{4D Hilbert tensor}},
     xtick = {0, 60},
    ytick={1, 1e-10},
    yticklabels={\small{$1$}, \small{$10^{-10}$}},
    x label style={at={(axis description cs:0.5,0)},anchor=north},
    xlabel={\small{rank}},
    ylabel={\small{$\|\mathcal{H}-\mathcal{H}_r\|_F/\|\mathcal{H}\|_F$}},
    legend pos=north west,
    ymajorgrids=true,
    grid style=dashed,
    width=.47\linewidth,
    legend pos = north east,
    mark size = 1.2pt,
]

\addplot [color = RoyalPurple, style = thick] table [col sep=comma, x index = 0, y index = 1] {E_4d_hilbert.txt};
\addplot [color = Emerald, mark = triangle*, mark size= 1.8pt] table [col sep=comma, x index = 0, y index = 2] {E_4d_hilbert.txt};
  \addplot [color = LimeGreen, mark = diamond*, mark size= 1.8pt] table [col sep=comma, x index = 0, y index = 3] {E_4d_hilbert.txt};
  \addplot [color = NavyBlue, mark = square*] table [col sep=comma, x index = 0, y index = 4] {E_4d_hilbert.txt};
     \legend{\scriptsize{HOSVD}, \scriptsize{RHOSVD},\scriptsize{RSTHOSVD},\scriptsize{MLN}}
\end{semilogyaxis}
\end{tikzpicture}
\caption{Frobenius error of approximation of the 3D Hilbert tensor: $\mathcal{H}(i,j,k) = \frac{1}{i+j+k-2}$ (left) and 4D Hilbert tensor: $\mathcal{H}(i,j,k,\ell) = \frac{1}{i+j+k+ \ell-3}$ (right).}
    \label{fig:Hilbert tensors}
\end{figure}
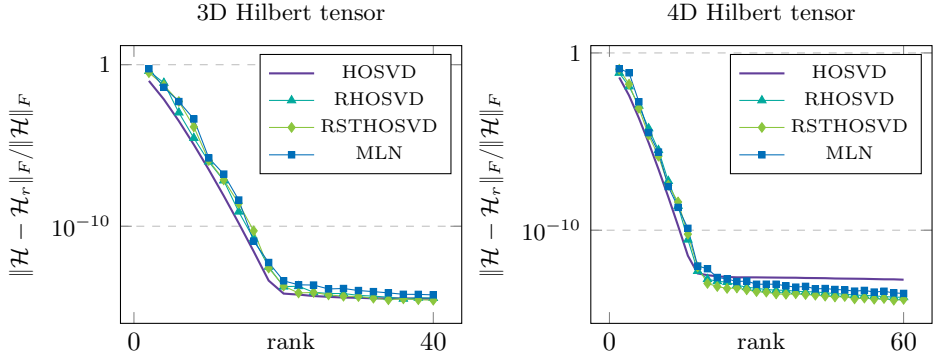
We remark that the Tucker 
approximations are only practical for such values of $d$, 
since for larger values
the storage cost for the core tensor can easily become 
the bottleneck.\\
To evaluate the execution time of the MLN method, in the next experiments, we compare it against the RHOSVD and RSTHOSVD.
In Figure \ref{fig:time_comparison_variable_ranks} we show the performances of these 3 methods on 3D tensors of size $100\times 100 \times 100$ and 4D tensors of size $70\times 70 \times 70 \times 70$ varying the multilinear rank of the approximants, while in Figure \ref{fig:time_comparison_variable_sizes} we fix the multilinear rank of the approximants and we increase the size of the problem.
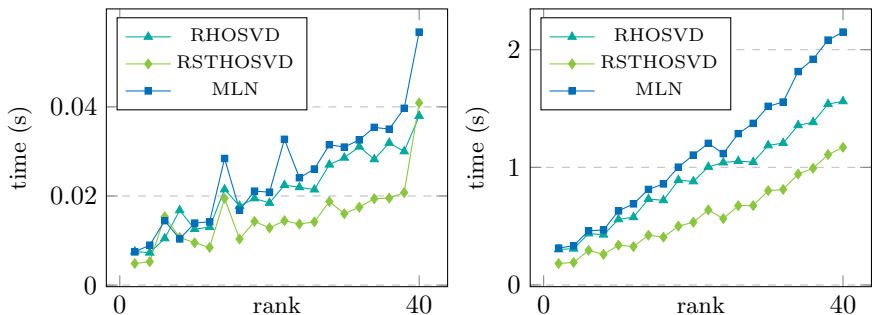
\begin{figure}
\centering\begin{tikzpicture}
\begin{axis}[
    x label style={at={(axis description cs:0.5,0)},anchor=north},
    xlabel={\small{rank}},
    ylabel={\small{time (s)}},
    legend pos=north west,
    xtick = {0,40},
    ytick = {0, 0.02, 0.04},
    yticklabel style={
        /pgf/number format/fixed,
        /pgf/number format/precision=2
      },
    scaled ticks=false,
    ymajorgrids=true,
    grid style=dashed,
    width=.47\linewidth,
    legend pos = north west,
    mark size = 1.2pt,
    restrict x to domain=0:40
]
    \addplot [color = Emerald, mark = triangle*, mark size= 1.8pt] table [col sep=comma, x index = 0, y index = 1] {time_comparison_variable_rank_3D.txt};
  \addplot [color = LimeGreen, mark = diamond*, mark size= 1.8pt] table [col sep=comma, x index = 0, y index = 2] {time_comparison_variable_rank_3D.txt};
  \addplot [color = NavyBlue, mark = square*] table [col sep=comma, x index = 0, y index = 3] {time_comparison_variable_rank_3D.txt};
\node[anchor=west, rotate=0] at (axis cs:-1,5e-13) {};
  \legend{\scriptsize{RHOSVD}, \scriptsize{RSTHOSVD}, \scriptsize{MLN} }

\end{axis}
\end{tikzpicture}~\begin{tikzpicture}
\begin{axis}[
    x label style={at={(axis description cs:0.5,0)},anchor=north},
    xlabel={\small{rank}},
    ylabel={\small{time (s)}},
    legend pos=north west,
    xtick = {0,40},
    ymajorgrids=true,
    grid style=dashed,
    width=.47\linewidth,
    legend pos = north west,
    mark size = 1.2pt,
    restrict x to domain=0:40
]
    \addplot [color = Emerald, mark = triangle*, mark size= 1.8pt] table [col sep=comma, x index = 0, y index = 1] {time_comparison_variable_rank_4D.txt};
  \addplot [color = LimeGreen, mark = diamond*, mark size= 1.8pt] table [col sep=comma, x index = 0, y index = 2] {time_comparison_variable_rank_4D.txt};
  \addplot [color = NavyBlue, mark = square*] table [col sep=comma, x index = 0, y index = 3] {time_comparison_variable_rank_4D.txt};
\node[anchor=west, rotate=0] at (axis cs:-1,5e-13) {};
  \legend{\scriptsize{RHOSVD}, \scriptsize{RSTHOSVD}, \scriptsize{MLN} }
\end{axis}
\end{tikzpicture}
\caption{Comparison of Tucker approximation methods in terms of computing time on 3D tensors (left) and 4D tensors (right) of fixed size.}
\label{fig:time_comparison_variable_ranks}
\end{figure}

\begin{figure}
\centering\begin{tikzpicture}
\begin{axis}[
    x label style={at={(axis description cs:0.5,0)},anchor=north},
    xlabel={\small{$n$}},
    ylabel={\small{time (s)}},
    legend pos=north west,
    xtick = {70,250},
    yticklabel style={
        /pgf/number format/fixed,
        /pgf/number format/precision=2
      },
    scaled ticks=false,
    ymajorgrids=true,
    grid style=dashed,
    width=.47\linewidth,
    legend pos = north west,
    mark size = 1.2pt,
    restrict x to domain=65:250
]
    \addplot [color = Emerald, mark = triangle*, mark size= 1.8pt] table [col sep=comma, x index = 0, y index = 1] {time_comparison_variable_size_3D.txt};
  \addplot [color = LimeGreen, mark = diamond*, mark size= 1.8pt] table [col sep=comma, x index = 0, y index = 2] {time_comparison_variable_size_3D.txt};
  \addplot [color = NavyBlue, mark = square*] table [col sep=comma, x index = 0, y index = 3] {time_comparison_variable_size_3D.txt};
\node[anchor=west, rotate=0] at (axis cs:-1,5e-13) {};
  \legend{\scriptsize{RHOSVD}, \scriptsize{RSTHOSVD}, \scriptsize{MLN} }
\end{axis}
\end{tikzpicture}~\begin{tikzpicture}
\begin{axis}[
    x label style={at={(axis description cs:0.5,0)},anchor=north},
    xlabel={\small{$n$}},
    ylabel={\small{time (s)}},
    legend pos=north west,
    xtick = {50,100},
    yticklabel style={
        /pgf/number format/fixed,
        /pgf/number format/precision=2
      },
    scaled ticks=false,
    ymajorgrids=true,
    grid style=dashed,
    width=.47\linewidth,
    legend pos = north west,
    mark size = 1.2pt,
    restrict x to domain=50:105
]
    \addplot [color = Emerald, mark = triangle*, mark size= 1.8pt] table [col sep=comma, x index = 0, y index = 1] {time_comparison_variable_size_4D.txt};
  \addplot [color = LimeGreen, mark = diamond*, mark size= 1.8pt] table [col sep=comma, x index = 0, y index = 2] {time_comparison_variable_size_4D.txt};
  \addplot [color = NavyBlue, mark = square*] table [col sep=comma, x index = 0, y index = 3] {time_comparison_variable_size_4D.txt};
\node[anchor=west, rotate=0] at (axis cs:-1,5e-13) {};
  \legend{\scriptsize{RHOSVD}, \scriptsize{RSTHOSVD}, \scriptsize{MLN} }
\end{axis}
\end{tikzpicture}
\caption{Comparison of Tucker approximation methods in terms of computing time on 3D tensors (left) and 4D tensors (right) with a fixed multilinear rank of approximation.}
\label{fig:time_comparison_variable_sizes}
\end{figure}
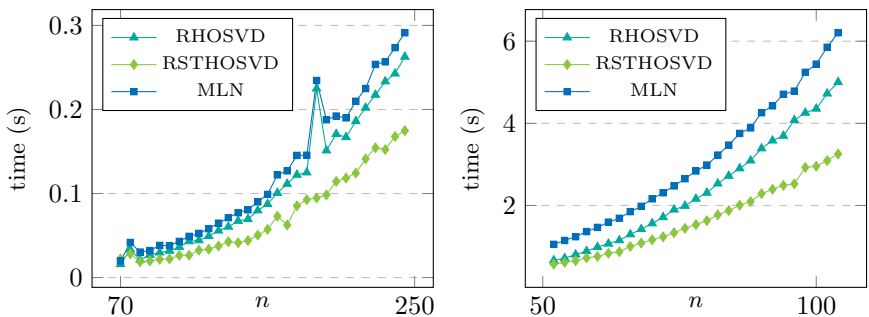

As we can see the computing time of MLN is slightly higher than that of RHOSVD and RSTHOSVD. This is because the bulk of the algorithm lies in the sketching procedure for not structured tensors and MLN requires oversampling. 

 In the next experiment, we investigate 
the effect of $d$ on the accuracy of the approximation.
In terms of accuracy, our analysis suggests a linear correlation between algorithmic error and the singular values of the tensor's matricizations and an exponential relationship with the norm of the projections $\mathcal{P}_k$, as shown for instance in Theorem~\ref{theorem: stability}. In particular,
since we do not expect the norm of the projectors $\widetilde{\mathcal P}_k$ to be influenced by $d$, 
the bound would predict an exponential growth of the constant with respect to $d$, because of the term 
$(1 + \tilde \tau)^d$. 

Figure \ref{fig:dimension comparison} shows 
a slight degradation of the quality of 
the approximation as $d$ increases, but this does 
not severely impact the performances (in particular 
in the case of exponential decay of the 
$\sigma_i$). 
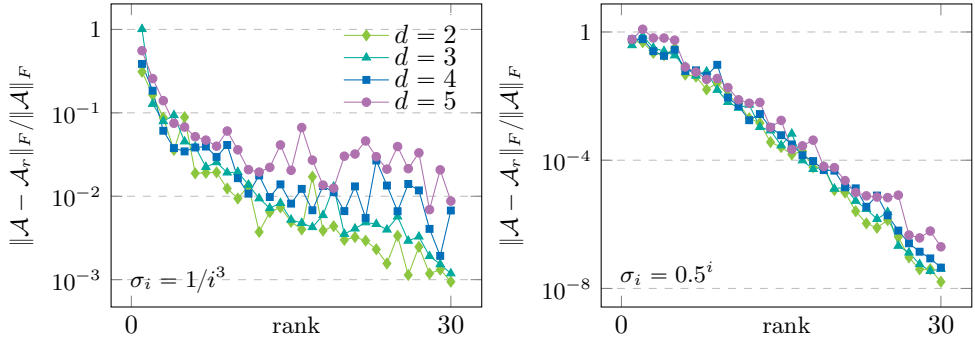
\begin{figure}
\centering\begin{tikzpicture}
\begin{semilogyaxis}[
    title={},
    x label style={at={(axis description cs:0.5,0)},anchor=north},
    xlabel={\small{rank}},
    ylabel={\small{$\|\mathcal{A}-\mathcal{A}_r\|_F/\|\mathcal{A}\|_F$}},
    xtick = {0, 30},
    ytick={1, 1e-1, 1e-2, 1e-3},
    yticklabels={\small{$1$},\small{$10^{-1}$},\small{$10^{-2}$}, \small{$10^{-3}$}},
    legend style={fill=none,draw=white, row sep=-4pt},
    ymajorgrids=true,
    grid style=dashed,
    width=.5\linewidth,
    legend pos = north east,
    mark size = 1.2pt,
    restrict x to domain=0:40
]

\addplot [color = LimeGreen, mark = diamond*, mark size= 2pt] table [col sep=comma, x index = 0, y index = 1] {E_dimension_dependence_cubic.txt};
    \addplot  [color = Emerald, mark = triangle*, mark size= 1.8pt] table [col sep=comma, x index = 0, y index = 2] {E_dimension_dependence_cubic.txt};
    \addplot [color = NavyBlue, mark = square*,mark size= 1.3pt] table [col sep=comma, x index = 0, y index = 3] {E_dimension_dependence_cubic.txt};
    \addplot [color = Orchid, mark = *, mark size= 1.6pt] table [col sep=comma, x index = 0, y index = 4] {E_dimension_dependence_cubic.txt};
    \node[anchor=west, rotate=0] at (axis cs:-1,1e-3) {\small{$ \sigma_i=1/i^3$}};
    \legend{$d=2$,$d=3$,$d=4$,$d=5$}
    
\end{semilogyaxis}
\end{tikzpicture}~\begin{tikzpicture}
\begin{semilogyaxis}[
    title={},
     xtick = {0, 30},
    ytick={1,1e-4, 1e-8},
    yticklabels={\small{$1$},\small{$10^{-4}$}, \small{$10^{-8}$}},
    x label style={at={(axis description cs:0.5,0)},anchor=north},
    xlabel={\small{rank}},
    ylabel={\small{$\|\mathcal{A}-\mathcal{A}_r\|_F/\|\mathcal{A}\|_F$}},
    legend pos=north west,
    ymajorgrids=true,
    grid style=dashed,
    width=.5\linewidth,
    legend pos = north east,
    mark size = 1.2pt,
]

\addplot [color = LimeGreen, mark = diamond*, mark size= 2pt] table [col sep=comma, x index = 0, y index = 1] {E_dimension_exponential.txt};
    \addplot  [color = Emerald, mark = triangle*, mark size= 1.8pt] table [col sep=comma, x index = 0, y index = 2] {E_dimension_exponential.txt};
    \addplot [color = NavyBlue, mark = square*,mark size= 1.3pt] table [col sep=comma, x index = 0, y index = 3] {E_dimension_exponential.txt};
    \addplot [color = Orchid, mark = *, mark size= 1.6pt] table [col sep=comma, x index = 0, y index = 4] {E_dimension_exponential.txt};
    \node[anchor=west, rotate=0] at (axis cs:-1,3e-8) {\small{$ \sigma_i=0.5^i$}};

\end{semilogyaxis}
\end{tikzpicture}
\caption{Multilinear Nystr\"om tested for tensors of varying dimension, $d=2,3,4,5$ but same decay rate: $\sigma_i = 1/i^3$ (left) and $\sigma_i = 0.5^i$ (right).}
    \label{fig:dimension comparison}
\end{figure}

In fact, the growth is far from 
the exponential one predicted by Theorem~\ref{thm : deterministic_bound}. 
We now perform another experiment aimed 
at understanding if Theorem~\ref{thm : deterministic_bound} is descriptive of the worst case behavior. To accomplish this, we select non-Gaussian 
sketching in an unfavorable setting.  Specifically, we use SRHT matrices for the sketchings. The tensor 
$\mathcal A = \mathcal S \times_{i = 1}^d Q_i^T$
 is constructed
as follows: instead of employing different Haar distributed orthogonal matrices for the $Q_i$, each $Q_i$ is set equal to the same $2\times 2$ block diagonal matrix $\mathrm{diag}(I_{7}, U)$, where $I_{7}$ is the $7\times 7$ identity matrix and $U\in \mathbb{R}^{25\times 25}$ is a Haar distributed orthogonal matrix. These $Q_i$ are known to be a difficult example for SRHT matrices \cite{SRHT_difficult, SRHT_difficult_Yuji}.
Regarding $\mathcal{S}$, we use a 4-dimensional superdiagonal tensor with exponential decay in the $\sigma_i$ of rate 0.3. Note that by construction such tensor is symmetric with respect to each mode. Therefore, setting the SRHT matrices $X_i$ and $Y_i$ equal to the same $X$ and $Y$ we expect that each projection contributes almost equally to the total error.

We report the numerical results in Table \ref{tab: exponential growth}, and consistently use the following notation, for $k =1, \dots, 4$: 
\[
E_k = \frac{|| \mathcal{T}-\mathcal{T}\times_{i=1}^k \mathcal{P}_i||_F}{|| \mathcal{T}-\mathcal{T}_r||_F}.
\]
In view of Theorem~\ref{thm : deterministic_bound}, 
we would expect that 
$E_{j+1} \approx (\tau + 1) E_j$ (assuming a worst-case behavior is encountered for all modes). The experiment demonstrates that the error growth is indeed exponential in $d$; in addition, the results 
in Table~\ref{tab: exponential growth} 
show that $1 + \tau$ describes the order of magnitude of 
$E_{j+1} / E_j$, as expected. To further highlight this, 
we report in Table~\ref{tab: exponential growth2}
the factors $1 + \tau$ and the error amplifications 
$E_{j+1} / E_j$. We remark that this example is not 
particularly meaningful from the low-rank approximation perspective: the errors 
are large and the low-rank 
approximations obtained of little practical use. We only 
include it to discuss whether Theorem~\ref{thm : deterministic_bound} gives an accurate description of the worst-case scenario.

Recall that this is not a limitation in practice: the 
method and the analysis are only of interest for moderate 
$d$, whereas for situations involving high dimensions, we 
suggest to look for alternatives that completely avoid 
the curse of dimensionality. Among these alternatives, the Streaming Tensor Train Approximation (STTA) \cite{kressner2022streaming} emerges as the closest in methodology to ours. It remains grounded in the GN framework, maintains streamability and one-pass capability, yet delivers an approximant in Tensor Train format.
\begin{table}[]
\caption{Frobenius relative error of approximation obtained by increasing the number of oblique projections used for the approximation. Each row 
of the table represents a different experiment; in this example, the sketchings are chosen in a particularly unfavorable manner in order to trigger the worst-case growth of the errors described in Theorem~\ref{thm : deterministic_bound}. }
    \centering
    \resizebox{\textwidth}{!}{
    \begin{tabular}{c|c|c|c|c|c|c}
    $\rho$ & $\tau$ & $|| \mathcal{T}-\mathcal{T}_r||_F$ & $E_1$ & $E_2$ & $E_3$ & $E_4$\\
    \hline
    $1.20\times 10^{-8}$ & $1.04\times 10^2$ & $4.51\times 10^{-9}$ & $9.67\times 10^0$ & $4.16\times 10^1$ & $6.38 \times 10^2$ & $2.02\times 10^4$\\
    \hline 
    $1.45\times 10^{-3}$ & $1.49\times 10^3$ & $4.51\times 10^{-9}$ & $3.63\times 10^7$ & $6.93\times 10^9$ & $2.45 \times 10^{12}$ & $9.27\times 10^{14}$\\
     \hline 
    $1.01\times 10^{-8}$ & $1.12\times 10^7$ & $4.51\times 10^{-9}$ & $4.16\times 10^6$ & $2.27\times 10^{12}$ & $6.12 \times 10^{18}$ & $1.94\times 10^{25}$\\
    \hline 
    $1.46\times 10^{-3}$ & $1.04\times 10^2$ & $4.51\times 10^{-9}$ & $1.22\times 10^6$ & $4.38\times 10^{6}$ & $3.66 \times 10^{7}$ & $3.33\times 10^{8}$\\
     \hline 
    $2.40\times 10^{-8}$ & $2.23\times 10^3$ & $4.51\times 10^{-9}$ & $3.92\times 10^3$ & $5.71\times 10^{5}$ & $2.98 \times 10^{8}$ & $2.22\times 10^{11}$
    \end{tabular}
    }
    \label{tab: exponential growth}
\end{table}

\begin{table}[]
\caption{The table represents the ratio of the 
errors obtained projecting on the first $j+1$ modes and 
the first $j$; Each row 
of the table represents a different experiment; the sketchings are chosen in a particularly unfavorable manner to trigger the worst-case growth of the errors described in Theorem~\ref{thm : deterministic_bound}. }
    \centering \small 
    \begin{tabular}{c|ccc}
$1 + \tau$ &  $E_2 / E_1$ & $E_3 / E_2$ & $E_4 / E_3$ \\ \hline
$1.05 \times 10^2$ & $4.33\times 10^0$ & $1.54\times 10^1$ &  $3.17\times 10^1$ \\ 
$1.49 \times 10^3$ & $1.90 \times 10^2$ & $3.54 \times 10^2$ & $3.78 \times 10^2$ \\ 
$1.12 \times 10^7$ & $5.46 \times 10^5$ &$2.72 \times 10^7$ & $3.17 \times 10^6$ \\ 
$1.05 \times 10^2$ & $3.61\times 10^0$ & $8.35\times 10^0$ &$ 9.10\times 10^0$ \\ 
$2.23 \times 10^3$ & $1.96 \times 10^2$ &
$5.22 \times 10^2$ & $7.45 \times 10^2$ \\ 
    \end{tabular}
    \label{tab: exponential growth2}
\end{table}

\section{Conclusions}

This paper presents and analyzes multilinear Nystr\"om (MLN), an innovative algorithm for the low-rank compression of a tensor in Tucker format. The method is based on the matrix 
Nystr\"om method, and in particular on the generalized Nystr\"om presented in \cite{nakatsukasa2020fast}.

Two key distinctions set our method apart from existing techniques. Firstly, MLN is a single-pass and streamable algorithm as it requires only two-sided sketchings of the original tensor.
Secondly, the MLN algorithm eliminates the need for costly orthogonalizations as it is based on the generalized Nystr\"om method for matrices; a crucial advantage over traditional approaches based on the randomized SVD. 

In terms of accuracy, the method exhibits only a marginal deviation from the RHOSVD, while still maintaining its near-optimal approximation quality. This assertion is reinforced not only by our rigorous theoretical analysis but also by the results obtained through extensive experiments.

Another crucial aspect of the method is its stability. 
Even though similar ideas have been proposed in the past (see for instance \cite{caiafa2014stable}), we propose 
suitable modifications that can ensure the stability of 
the methodology; this is attained by carefully handling 
the pseudoinverses involved in the process. 
This stability further strengthens the applicability and practicality of our proposed approach.
Several research lines remain open and will be 
investigated in the future. Our work characterizes
the quality of the approximants based on the norm 
of certain matrices (%
see Theorem~\ref{thm : deterministic_bound}).

In principle, finding a-priori probabilistic bounds for such matrices allows us to select the best sketching in any given scenario. While deriving an a-priori bound proves to be relatively easy in certain instances (such as when dealing with Gaussian matrices, as discussed in our paper), the complexity of the analysis increases when adopting structured sketching techniques. These techniques, while highly beneficial for structured tensors, present a more intricate analytical challenge. Notably, substantial efforts are already underway in this direction.
Another intriguing avenue of research involves expanding the methodology outlined in this study to encompass a broader spectrum of tensor networks. Such an extension could potentially unlock novel insights and applications across a wider range of contexts within the realm of tensor-based computations.

\section*{Acknowledgments}
Alberto Bucci and Leonardo Robol are members of the INdAM Research group GNCS (Gruppo Nazionale di Calcolo Scientifico)
and have been supported by the PRIN 2022 Project ``Low-rank Structures and Numerical Methods in Matrix and Tensor Computations and their Application'' and 
by the MIUR Excellence Department Project awarded to the Department of Mathematics, University of Pisa, CUP I57G22000700001. The work of Leonardo Robol has been supported by the National Research Center in High Performance Computing, Big Data and Quantum Computing (CN1 -- Spoke 6).

\bibliographystyle{siamplain}
\bibliography{references}

\begin{thebibliography}{10}

\bibitem{randhosvdsurvey}
{\sc S.~Ahmadi-Asl, S.~Abukhovich, M.~G. Asante-Mensah, A.~Cichocki, A.~H.
  Phan, T.~Tanaka, and I.~Oseledets}, {\em {R}andomized {A}lgorithms for
  {C}omputation of {T}ucker {D}ecomposition and {H}igher {O}rder {SVD}
  ({HOSVD})}, IEEE Access, 9 (2021), pp.~28684--28706,
  \url{https://doi.org/10.1109/ACCESS.2021.3058103}.

\bibitem{tensor_toolbox}
{\sc B.~W. Bader and T.~G. Kolda}, {\em Algorithm 862: {MATLAB} tensor classes
  for fast algorithm prototyping}, ACM Trans. Math. Software, 32 (2006),
  pp.~635--653, \url{https://doi.org/10.1145/1186785.1186794}.

\bibitem{curse_of_dimensionality}
{\sc R.~E. Bellman}, {\em Adaptive control processes: A guided tour.},
  Princeton University Press,  (1961).

\bibitem{boutsidis2013improved}
{\sc C.~Boutsidis and A.~Gittens}, {\em Improved matrix algorithms via the
  subsampled randomized {H}adamard transform}, SIAM J. Matrix Anal. Appl., 34
  (2013), pp.~1301--1340, \url{https://doi.org/10.1137/120874540}.

\bibitem{SRHT_difficult}
{\sc C.~Boutsidis and A.~Gittens}, {\em Improved matrix algorithms via the
  subsampled randomized {H}adamard transform}, SIAM J. Matrix Anal. Appl., 34
  (2013), pp.~1301--1340, \url{https://doi.org/10.1137/120874540}.

\bibitem{caiafa2010generalizing}
{\sc C.~F. Caiafa and A.~Cichocki}, {\em Generalizing the column--row matrix
  decomposition to multi-way arrays}, Linear Algebra and its Applications, 433
  (2010), pp.~557--573, \url{https://doi.org/10.1016/j.laa.2010.03.020}.

\bibitem{caiafa2014stable}
{\sc C.~F. Caiafa and A.~Cichocki}, {\em Stable, robust, and super fast
  reconstruction of tensors using multi-way projections}, IEEE Trans. Signal
  Process., 63 (2015), pp.~780--793,
  \url{https://doi.org/10.1109/TSP.2014.2385040}.

\bibitem{R-ST-HOSVD}
{\sc M.~Che and Y.~Wei}, {\em Randomized algorithms for the approximations of
  {T}ucker and the tensor train decompositions}, Adv. Comput. Math., 45 (2019),
  pp.~395--428, \url{https://doi.org/10.1007/s10444-018-9622-8}.

\bibitem{power_tucker}
{\sc M.~Che, Y.~Wei, and H.~Yan}, {\em The computation of low multilinear rank
  approximations of tensors via power scheme and random projection}, SIAM J.
  Matrix Anal. Appl., 41 (2020), pp.~605--636,
  \url{https://doi.org/10.1137/19M1237016}.

\bibitem{hosvd}
{\sc L.~De~Lathauwer, B.~De~Moor, and J.~Vandewalle}, {\em A multilinear
  singular value decomposition}, SIAM J. Matrix Anal. Appl., 21 (2000),
  pp.~1253--1278, \url{https://doi.org/10.1137/S0895479896305696}.

\bibitem{tropp_GN}
{\sc W.~Dong, G.~Yu, L.~Qi, and X.~Cai}, {\em Practical sketching algorithms
  for low-rank {T}ucker approximation of large tensors}, J. Sci. Comput., 95
  (2023), pp.~Paper No. 52, 26,
  \url{https://doi.org/10.1007/s10915-023-02172-y},
  \url{https://doi.org/10.1007/s10915-023-02172-y}.

\bibitem{other_tucker}
{\sc W.~Dong, G.~Yu, L.~Qi, and X.~Cai}, {\em Practical sketching algorithms
  for low-rank {T}ucker approximation of large tensors}, J. Sci. Comput., 95
  (2023), pp.~Paper No. 52, 26,
  \url{https://doi.org/10.1007/s10915-023-02172-y}.

\bibitem{Tensor_hoid}
{\sc P.~Drineas and M.~W. Mahoney}, {\em A randomized algorithm for a
  tensor-based generalization of the singular value decomposition}, Linear
  Algebra Appl., 420 (2007), pp.~553--571,
  \url{https://doi.org/10.1016/j.laa.2006.08.023}.

\bibitem{svd}
{\sc G.~H. Golub and C.~F. Van~Loan}, {\em Matrix computations, third
  edition.}, Johns Hopkins University Press,  (1996).

\bibitem{CUR_decomposition}
{\sc S.~A. Goreinov, E.~E. Tyrtyshnikov, and N.~L. Zamarashkin}, {\em A theory
  of pseudoskeleton approximations}, Linear Algebra Appl., 261 (1997),
  pp.~1--21, \url{https://doi.org/10.1016/S0024-3795(96)00301-1}.

\bibitem{grasedyck2010hierarchical}
{\sc L.~Grasedyck}, {\em Hierarchical singular value decomposition of tensors},
  SIAM J. Matrix Anal. Appl., 31 (2009/10), pp.~2029--2054,
  \url{https://doi.org/10.1137/090764189}.

\bibitem{grasedyck2013literature}
{\sc L.~Grasedyck, D.~Kressner, and C.~Tobler}, {\em A literature survey of
  low-rank tensor approximation techniques}, GAMM-Mitt., 36 (2013), pp.~53--78,
  \url{https://doi.org/10.1002/gamm.201310004}.

\bibitem{halko2011finding}
{\sc N.~Halko, P.~G. Martinsson, and J.~A. Tropp}, {\em Finding structure with
  randomness: probabilistic algorithms for constructing approximate matrix
  decompositions}, SIAM Rev., 53 (2011), pp.~217--288,
  \url{https://doi.org/10.1137/090771806}.

\bibitem{tucker}
{\sc F.~L. Hitchcock}, {\em The expression of a tensor or a polyadic as a sum
  of products}, J. Math. Phys, 6 (1927), pp.~164--189,
  \url{https://doi.org/10.1002/sapm192761164}.

\bibitem{Kolda}
{\sc R.~Jin, T.~G. Kolda, and R.~Ward}, {\em Faster {J}ohnson-{L}indenstrauss
  transforms via {K}ronecker products}, Inf. Inference, 10 (2021),
  pp.~1533--1562, \url{https://doi.org/10.1093/imaiai/iaaa028}.

\bibitem{tensordecomp}
{\sc T.~G. Kolda and B.~W. Bader}, {\em Tensor decompositions and
  applications}, SIAM Rev., 51 (2009), pp.~455--500,
  \url{https://doi.org/10.1137/07070111X}.

\bibitem{kressner2023analysis}
{\sc D.~Kressner and B.~Plestenjak}, {\em Analysis of a class of randomized
  numerical methods for singular matrix pencils}, arXiv preprint
  arXiv:2305.13118,  (2023).

\bibitem{kressner2022streaming}
{\sc D.~Kressner, B.~Vandereycken, and R.~Voorhaar}, {\em Streaming tensor
  train approximation}, arXiv preprint arXiv:2208.02600,  (2022).

\bibitem{randomization_advantajes}
{\sc P.-G. Martinsson and J.~A. Tropp}, {\em Randomized numerical linear
  algebra: foundations and algorithms}, Acta Numer., 29 (2020), pp.~403--572,
  \url{https://doi.org/10.1017/s0962492920000021}.

\bibitem{minster_li_ballard}
{\sc R.~Minster, Z.~Li, and G.~Ballard}, {\em Parallel randomized tucker
  decomposition algorithms}, arXiv preprint arXiv:2211.13028,  (2022).

\bibitem{Saibaba_Minster_Arvind}
{\sc R.~Minster, A.~K. Saibaba, and M.~E. Kilmer}, {\em Randomized algorithms
  for low-rank tensor decompositions in the {T}ucker format}, SIAM J. Math.
  Data Sci., 2 (2020), pp.~189--215, \url{https://doi.org/10.1137/19M1261043}.

\bibitem{nakatsukasa2020fast}
{\sc Y.~Nakatsukasa}, {\em Fast and stable randomized low-rank matrix
  approximation}, arXiv preprint arXiv:2009.11392,  (2020).

\bibitem{common_practice}
{\sc Y.~Nakatsukasa and N.~J. Higham}, {\em Backward stability of iterations
  for computing the polar decomposition}, SIAM J. Matrix Anal. Appl., 33
  (2012), pp.~460--479, \url{https://doi.org/10.1137/110857544}.

\bibitem{SRHT_difficult_Yuji}
{\sc Y.~Nakatsukasa and T.~Park}, {\em Randomized low-rank approximation for
  symmetric indefinite matrices}, SIAM J. Matrix Anal. Appl., 44 (2023),
  pp.~1370--1392, \url{https://doi.org/10.1137/22M1538648}.

\bibitem{nakatsukasa2021fast}
{\sc Y.~Nakatsukasa and J.~A. Tropp}, {\em Fast \& accurate randomized
  algorithms for linear systems and eigenvalue problems}, arXiv preprint
  arXiv:2111.00113,  (2021).

\bibitem{oseledets2011tensor}
{\sc I.~V. Oseledets}, {\em Tensor-train decomposition}, SIAM J. Sci. Comput.,
  33 (2011), pp.~2295--2317, \url{https://doi.org/10.1137/090752286}.

\bibitem{oseledets2008tucker}
{\sc I.~V. Oseledets, D.~Savostianov, and E.~E. Tyrtyshnikov}, {\em Tucker
  dimensionality reduction of three-dimensional arrays in linear time}, SIAM
  Journal on Matrix Analysis and Applications, 30 (2008), pp.~939--956,
  \url{https://doi.org/10.1137/060655894}.

\bibitem{rokhlin2008fast}
{\sc V.~Rokhlin and M.~Tygert}, {\em A fast randomized algorithm for
  overdetermined linear least-squares regression}, Proc. Natl. Acad. Sci. USA,
  105 (2008), pp.~13212--13217, \url{https://doi.org/10.1073/pnas.0804869105}.

\bibitem{HOID}
{\sc A.~K. Saibaba}, {\em H{OID}: higher order interpolatory decomposition for
  tensors based on {T}ucker representation}, SIAM J. Matrix Anal. Appl., 37
  (2016), pp.~1223--1249, \url{https://doi.org/10.1137/15M1048628}.

\bibitem{Another_Tucker}
{\sc Y.~Sun, Y.~Guo, C.~Luo, J.~Tropp, and M.~Udell}, {\em Low-rank {T}ucker
  approximation of a tensor from streaming data}, SIAM J. Math. Data Sci., 2
  (2020), pp.~1123--1150, \url{https://doi.org/10.1137/19M1257718}.

\bibitem{tropp2011improved}
{\sc J.~A. Tropp}, {\em Improved analysis of the subsampled randomized
  {H}adamard transform}, Adv. Adapt. Data Anal., 3 (2011), pp.~115--126,
  \url{https://doi.org/10.1142/S1793536911000787}.

\bibitem{sthosvd}
{\sc N.~Vannieuwenhoven, R.~Vandebril, and K.~Meerbergen}, {\em A new
  truncation strategy for the higher-order singular value decomposition}, SIAM
  J. Sci. Comput., 34 (2012), pp.~A1027--A1052,
  \url{https://doi.org/10.1137/110836067}.

\bibitem{higherorderHMT}
{\sc G.~Zhou, A.~Cichocki, and S.~Xie}, {\em Decomposition of big tensors with
  low multilinear rank}, arXiv preprint arXiv:1412.1885,  (2014).

\end{thebibliography}
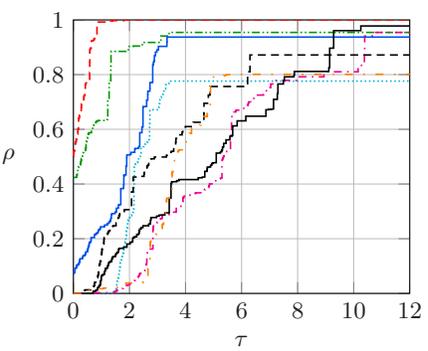
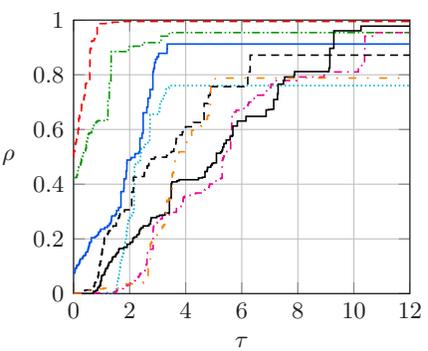
\end{document}